\documentclass{gtpart}

\usepackage{graphicx}  

\usepackage{thm-restate}
\usepackage{hyperref}
\usepackage{cleveref}
\usepackage{todonotes}
\usepackage{booktabs}

\usepackage{lscape}
\usepackage{multirow}
\usepackage{chemarrow}
\makeatletter
\renewcommand{\section}{%
\@startsection{section}{1}{\z@}%
 {-3.5ex \@plus -1ex \@minus -.2ex}%
                                   {2.3ex \@plus.2ex}%
{\reset@font
\large\bfseries}}
\renewcommand{\thesection}{\@arabic\c@section}
\makeatother

\newcommand{\rmv}{\smallsetminus}
\newcommand{\bd}{\partial}
\newcommand{\set}[1]{ \left\{ #1 \right\} }

\newcommand{\AC}{\mathcal{A}}
\newcommand{\OAC}{\mathcal{A}^O}
\newcommand{\T}{\mathcal{T}}

\title[Generalized Crossing Changes between GOF-knots]
{COMPLETE CLASSIFICATION OF \\GENERALIZED CROSSING CHANGES BETWEEN GOF-KNOTS}

\author{Kai Ishihara and Matt Rathbun}

\address{Kai Ishihara
\newline Yamaguchi University
\newline 1677-1 YoshidaYamaguchi 753-8513
\newline Japan
\newline e-mail: kisihara@yamaguchi-u.ac.jp}
\address{Matt Rathbun
\newline California State University, Fullerton
\newline Fullerton, CA
\newline USA
\newline e-mail: mrathbun@fullerton.edu}
\renewcommand{\proofname}{\indent\sc Proof.} 
\newtheorem{thm}{\indent\bf Theorem}[section] 
\newtheorem{lem}[thm]{\indent\bf Lemma}
\newtheorem{cor}[thm]{\indent\bf Corollary}
\theoremstyle{definition}
\newtheorem{rem}[thm]{\indent\sc Remark}
\newtheorem{example}{\indent\sc Example}

\begin{document}
\footnote{
\newline
This work was supported by JSPS KAKENHI Grant Number 17K14190.}

%

\begin{abstract}    

We analyze all monodromies of genus one fibered knots that possess clean or once-unclean arcs, and use this to determine all manifolds containing genus one fibered knots with generalized crossing changes resulting in another genus one fibered knot, and classify all such generalized crossing changes between two genus one, fibered knots.
\end{abstract}
\maketitle

\section{Introduction}
Every closed, orientable $3$-manifold contains a fibered knot, a knot whose exterior fibers over the circle with the knot bounding the fibers. A genus one fibered knot, or GOF-knot, is a fibered knot whose fiber is a once-punctured torus. 

A \emph{crossing circle} for a link 
$K$ is a circle $L$ that bounds a disk intersecting $K$ in two points with opposite orientations. 
We refer to the disk as a \emph{crossing disk}. 
Then, a \emph{generalized crossing change along $L$ of order $q$} is a $- \frac{1}{q}$ Dehn surgery on $L$, with $q \in \mathbb{Z} \rmv \set{0}$. Since $L$ bounds a disk, the ambient manifold  does not change, but the link may. 
When $q = \pm 1$, this is just an ordinary \emph{crossing change}.  

Let $K_1$ and $K_2$ be links 
in closed oriented $3$-manifolds $M_1$ and $M_2$, respectively. 
We say $K_1$ and $K_2$ are \emph{equivalent}, denoted $K_1\cong K_2$, if there is an orientation preserving homeomorphism from $M_1$ to $M_2$ which maps $K_1$ onto $K_2$. 
Let $L_1$ and $L_2$ be crossing circles for $K_1$ and $K_2$, respectively.
We say the generalized crossing change on $K_1$ along $L_1$ of order $q_1$ and that on $K_2$ along $L_2$ of order $q_2$ are \emph{equivalent} 
if $q_1=q_2$ and $K_1\cup L_1\cong K_2\cup L_2$.
Note that the resulting links are equivalent if the generalized crossing changes are equivalent. 
It is well known that any GOF-knot in the $3$-sphere $S^3$ is equivalent to a (left-hand or right-hand) trefoil knot or a figure eight knot 
 \cite{BurZieHeuK}. Morimoto \cite{MorGOFKLS} investigated how many GOF-knots are in a lens space, and Baker \cite{BakCG1FKLS} completes this investigation by giving a criterion for determining the exact number of GOF-knots in each lens space. 
It is also known that there is no classical (order $1$) crossing change between any two of the three GOF-knots in $S^3$. 
In this paper, we classify all generalized crossing changes between GOF-knots in any $3$-manifold.

For the classification of generalized crossing changes between GOF-knots, we introduce families of GOF-knots, see Fig. \ref{fig:Kklm}:
For any two integers $k$ and $\ell$, a knot $GOF(0;k,\ell)$ in a $3$-manifold $M(0;k,\ell)$ is obtained from the Borromean rings by $-k$ and $-\ell$ surgeries on two components.
For any integer $m$, a knot $GOF(1;m)$ in a $3$-manifold $M(1;m)$ is obtained from $L8n5$ in the Thistlethwaite link table (the mirror image of $8^3_{9}$ in the Rolfsen table\cite{RolKL}) by $0$ and $-m$ surgeries on the sub-link that is a $(2,4)$-torus link. $GOF(-1;m)$ is the mirror image of $GOF(1;-m)$.

\begin{figure}[h]
\begin{center}
\includegraphics[scale=.6]{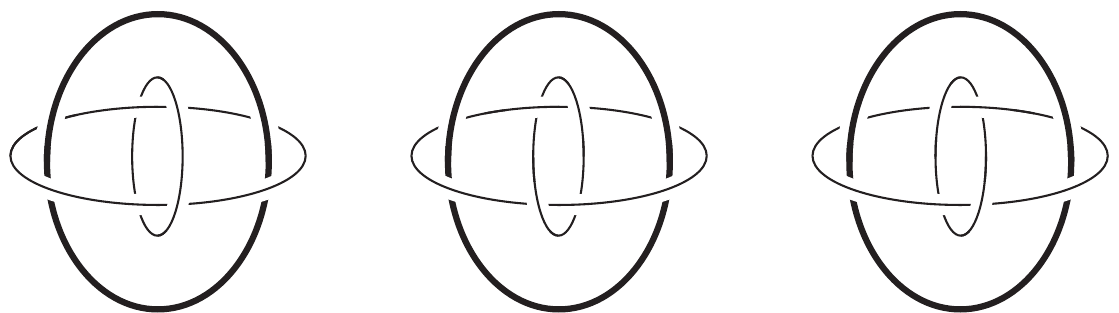}

\begin{picture}(400,0)(0,0)
\put(40,0){$GOF(0;k,\ell)$}
\put(80,77){$-\ell$}
\put(65,87){$-k$}
\put(160,0){$GOF(1;m)$}
\put(195,77){$-m$}
\put(184,87){$0$}
\put(280,0){$GOF(-1;m)$}
\put(300,87){$0$}
\put(310,77){$-m$}
\end{picture}
\caption{Families of GOF-knots.}
\label{fig:Kklm}
\end{center}
\end{figure}

\begin{restatable}{thm}{TheoremCrosingChanges}\label{thm:CrossingChanges}
Any order $q$ generalized crossing change between distinct GOF-knots is equivalent to one of the following for some integer $n$, see Fig. \ref{fig:crossingchangeGOF}.
\begin{itemize}
\item[(1)] $q=\pm2$, a generalized crossing change between $GOF(0;n,1)$ and $GOF(0;n,-1)$, 
\item[(2)] $q=\pm1$, a (classical) crossing change between $GOF(0;n,2)$ and $GOF(0;n,-2)$,  
\item[(3)] $q=\pm1$, a (classical) crossing change between $GOF(1;n-2)$ and $GOF(-1;n+2)$.  
\end{itemize}
\end{restatable}

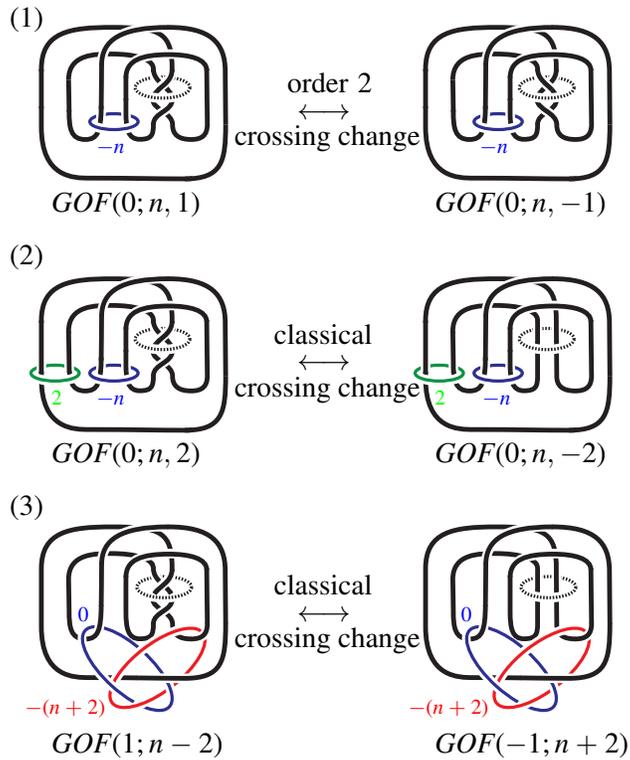
\begin{figure}[h]
\begin{center}
\includegraphics[scale=.6]{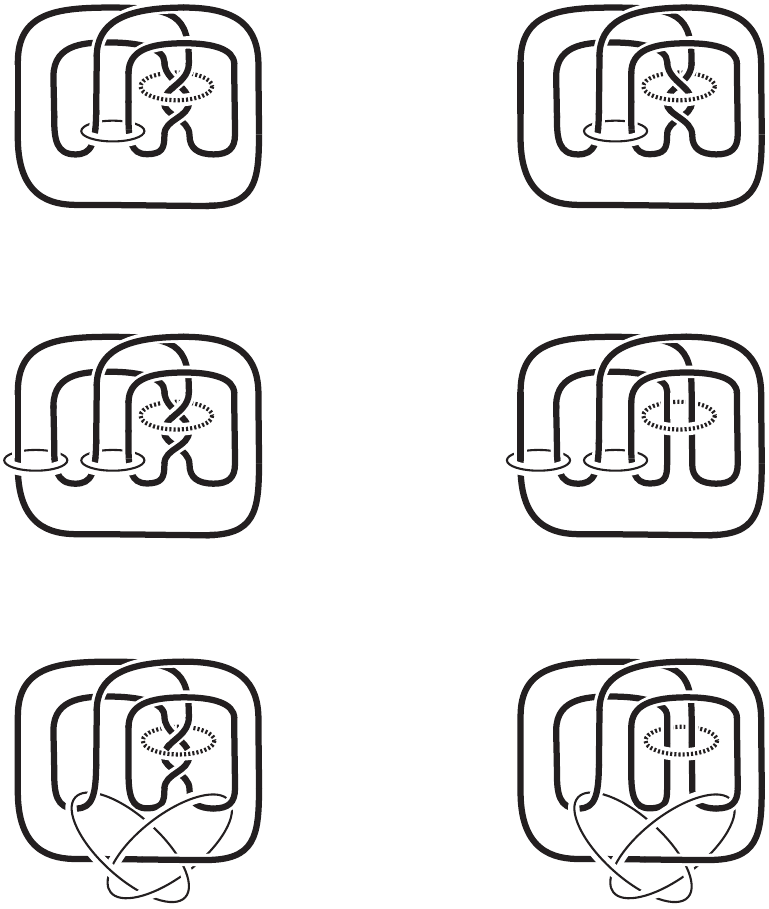}

\begin{picture}(400,0)(0,0)
\put(70,275){(1)}
\put(85,205){$GOF(0;n,1)$}
\put(102,230){\scriptsize $-n$}
\put(247,230){\scriptsize $-n$}
\put(230,205){$GOF(0;n,-1)$}
\put(178,240){$\longleftrightarrow$}
\put(175,250){order $2$}
\put(155,230){crossing change}
\put(70,185){(2)}
\put(85,110){$GOF(0;n,2)$}
\put(85,135){\scriptsize $2$}
\put(102,135){\scriptsize $-n$}
\put(230,110){$GOF(0;n,-2)$}
\put(230,135){\scriptsize $2$}
\put(248,135){\scriptsize $-n$}
\put(178,145){$\longleftrightarrow$}
\put(170,155){classical}
\put(155,135){crossing change}
\put(70,90){(3)}
\put(85,0){$GOF(1;n-2)$}
\put(75,18){\scriptsize $-(n+2)$}
\put(135,18){\scriptsize $0$}
\put(220,18){\scriptsize $-(n+2)$}
\put(280,18){\scriptsize $0$}
\put(230,0){$GOF(-1;n+2)$}
\put(178,50){$\longleftrightarrow$}
\put(170,60){classical}
\put(155,40){crossing change}
\end{picture}
\caption{Generalized crossing changes between GOF-knots.}
\label{fig:crossingchangeGOF}
\end{center}
\end{figure}

As a result of identifying the manifolds $M(0; k, \ell)$ and $M(\pm1; m)$ in which $GOF(0; k, \ell)$ and $GOF(\pm1; m)$ sit (Theorem \ref{thm:Manifolds}), we have the following corollary.

\begin{restatable}{cor}{ManifoldsForGOFKnotsWithGeneralizedCrossingChanges} \leavevmode
\label{cor:ManifoldsForGOFKnotsWithGeneralizedCrossingChanges}
\begin{enumerate}
\item Every GOF-knot with a non-classical generalized crossing change resulting in another GOF-knot is in $L(n, 1)$ for some $n \in \mathbb{Z}$.
\item Every GOF-knot with a classical crossing change resulting in another GOF-knot is in $L(2, 1) \, \# \, L(n, 1)$ for some $n \in \mathbb{Z}$, $L(4,\pm1)$, or a prism manifold.\end{enumerate}
\end{restatable}
Here $L(0,1)$ and $L(\pm1,1)$ refer to $S^2\times S^1$ and $S^3$, respectively, and
$ \# $ denotes the connected sum of two manifolds.

Non-classical generalized crossing changes (resp., classical crossing changes) between GOF-knots must occur at specific arcs in the genus one fiber surface, which are called clean and alternating arcs (resp., once-unclean and alternating arcs) (Lemma \ref{lem:AllCrossingChangesAlongArcs}).
When such arcs are present, we will be able to factorize monodromies in specific ways (Lemmas 
\ref{lemma:Alternating}, \ref{lemma:FactorizationClean} and \ref{lemma:FactorizationUnclean}) to prove Theorem \ref{thm:CrossingChanges}.
Further, we will identify when there are arcs that may give rise to inequivalent generalized crossing changes, and prove the following corollary.

\begin{restatable}{cor}{CorollaryGOFequivalence}\label{cor:GOFequivalence}
Any GOF-knot $K$ has at most two equivalence classes of generalized crossing changes which produce another GOF-knot.
Moreover, when $K$ has two generalized crossing changes, $K$ is equivalent to one of the following:
\begin{enumerate} 
\item \label{item:GOFS3} $GOF(0;1,-1)$ ($\cong GOF(0;-1,1)$ a figure-eight knot) in $S^3$.
\item \label{item:GOFL(2,1)} $GOF(0;\pm1,\pm2)$ ($\cong GOF(0;\pm2,\pm1)$) 
or \\$GOF(0;\pm1,\mp2)$ ($\cong GOF(0;\mp2,\pm1)$)  
in $L(2,1)$.
\item \label{item:GOFL(4,1)} $GOF(0;-4,-1)$ ($\cong GOF(-1;1)$) 
in $L(4,1)$.
\item[$(3)'$] $GOF(0;4,1)$ ($\cong GOF(1;-1)$) 
in $L(4,-1)$.
\item \label{item:GOFL(2,1)sharpL(2,1)} $GOF(0;\pm2,\pm2)$ ($\cong GOF(\pm1;\mp2)$) or \\$GOF(0;2,-2)$ ($\cong GOF(0;-2,2)$) in $L(2,1)\, \sharp\, L(2,1)$.
\end{enumerate}
\end{restatable}

\begin{rem}
For the homeomorphism indicated in item (\ref{item:GOFL(4,1)}), see Fig. \ref{fig:euivL(4,1)}, and for the first homeomorphism indicated in item (\ref{item:GOFL(2,1)sharpL(2,1)}), see Fig. \ref{fig:equivL(2,1)L(2,1)}, both in the \nameref{section:Appendix}. 
\end{rem}

\begin{example}
The $3$-manifolds $S^3$, $L(4,1)$, and $L(2,1)\ \sharp\ L(2,1)$ each have three distinct GOF-knots up to equivalence. 
The manifold $L(2,1)$ has four distinct GOF-knots (two mirror pairs) up to equivalence. 
In each case, all these GOF-knots are related by generalized crossing changes, see Figs. \ref{fig:crossingchangeS^3},  \ref{fig:crossingchangeL(4,1)}, \ref{fig:crossingchangeL(2,1)L(2,1)},  and \ref{fig:crossingchangeL(2,1)}.  
\end{example}
%
\begin{figure}[h!]
\begin{center}
\includegraphics[scale=.6]{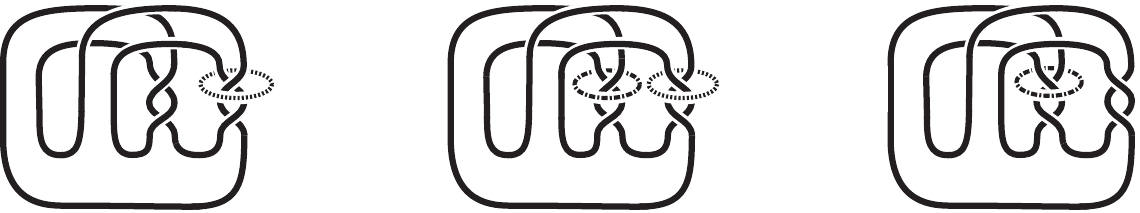}

\begin{picture}(400,0)(0,0)
\put(110,55){order $2$}
\put(240,55){order $2$}
\put(110,45){$\longleftrightarrow$}
\put(240,45){$\longleftrightarrow$}
\put(30,0){$GOF(0;1,1)$}
\put(25,80){(left-hand trefoil)}
\put(120,0){$GOF(0;1,-1)\cong GOF(0;-1,1)$}
\put(160,80){(figure-eight)}
\put(280,0){$GOF(0;-1,-1)$}
\put(280,80){(right-hand trefoil)}
\end{picture}
\caption{Generalized crossing changes between GOF-knots in $S^3$.}
\label{fig:crossingchangeS^3}
\end{center}
\end{figure}
\begin{figure}[h!]
\begin{center}
\includegraphics[scale=.6]{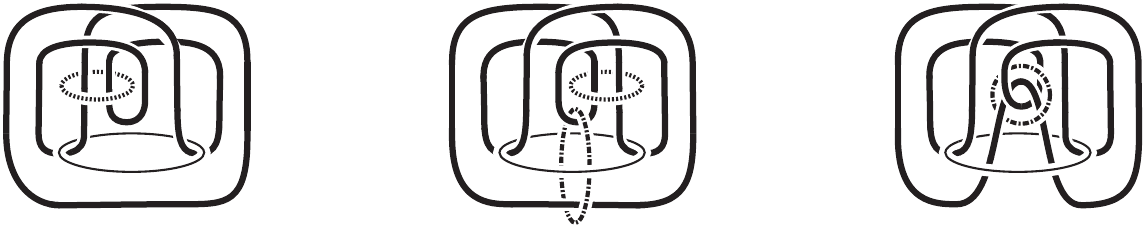}

\begin{picture}(400,0)(0,0)
\put(40,25){\scriptsize $4$}
\put(168,25){\scriptsize $4$}
\put(295,25){\scriptsize $4$}
\put(110,55){order $2$}
\put(230,55){classical}
\put(110,45){$\longleftrightarrow$}
\put(240,45){$\longleftrightarrow$}
\put(30,0){$GOF(0;-4,1)$}
\put(120,0){$GOF(0;-4,-1) \cong GOF(-1;1)$}
\put(290,0){$GOF(1;-3)$}
\end{picture}
\caption{Generalized crossing changes between GOF-knots in $L(4,1)$.}
\label{fig:crossingchangeL(4,1)}
\end{center}
\end{figure}
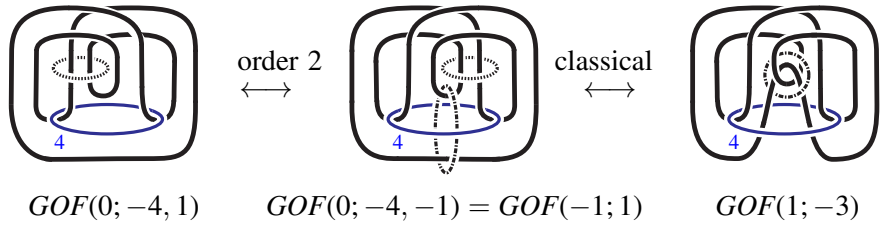

\begin{rem}
By translating the monodromies listed here into elements of $SL_2(\mathbb{Z})$, it can be shown that the knots in Fig. \ref{fig:crossingchangeL(4,1)} (from left to right) correspond to $K_1$, $K_2$, and $K_3$ from \cite{MorGOFKLS}. A particularly interesting observation, then, is that there is a crossing change between the knots $K_2$ and $K_3$ in  \cite{MorGOFKLS}. 
\end{rem}

\begin{figure}[h!]
\begin{center}
\includegraphics[scale=.6]{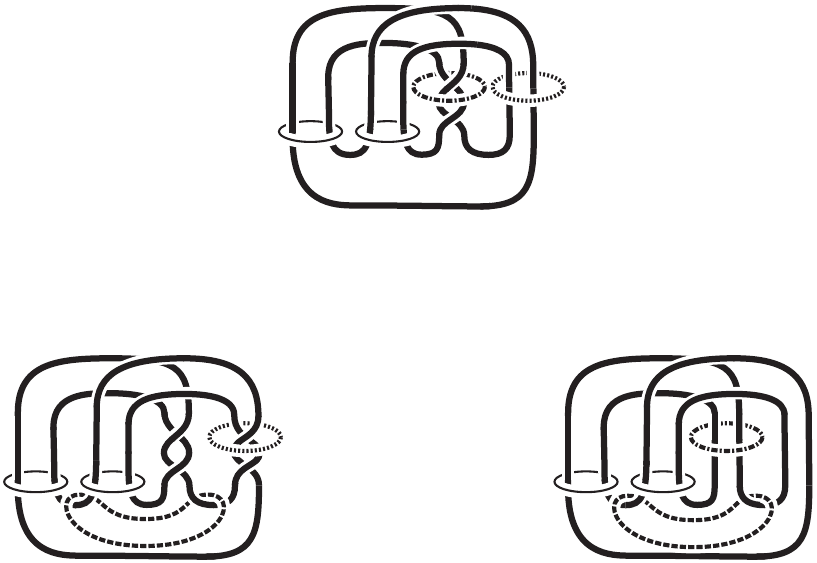}

\begin{picture}(400,0)(0,0)
\put(100,42){\scriptsize $2$}
\put(77,42){\scriptsize $2$}
\put(179,143){\scriptsize $2$}
\put(157,143){\scriptsize $2$}
\put(258,42){\scriptsize $2$}
\put(237,42){\scriptsize $2$}
\put(130,85){\rotatebox{45}{$\longleftrightarrow$}}
\put(230,85){\rotatebox{135}{$\longleftrightarrow$}}
\put(180,45){$\longleftrightarrow$}
\put(40,0){$GOF(0;2,2)\cong GOF(1;-2)$}
\put(120,105){$GOF(0;2,-2) \cong GOF(0;-2,2)$}
\put(200,0){$GOF(-1;2) \cong GOF(0;-2,-2)$}
\end{picture}
\caption{Classical crossing changes between GOF-knots in $L(2,1)\, \sharp\, L(2,1)$.}
\label{fig:crossingchangeL(2,1)L(2,1)}
\end{center}
\end{figure}
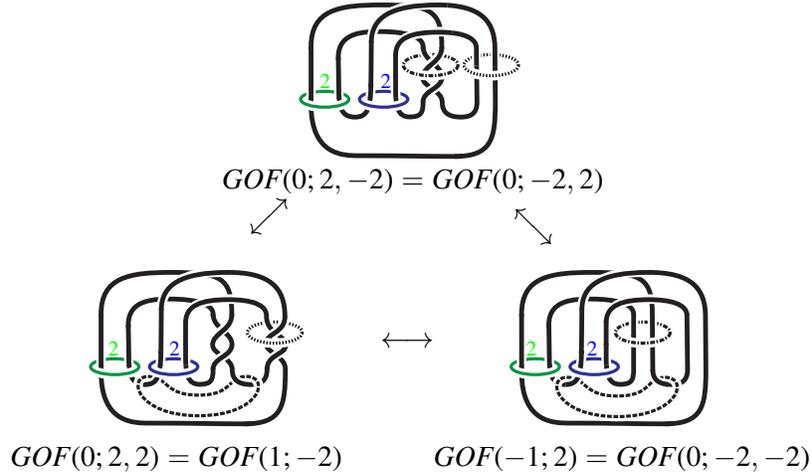

\begin{figure}[h!]
\begin{center}
\vspace{1cm}
\includegraphics[scale=.6]{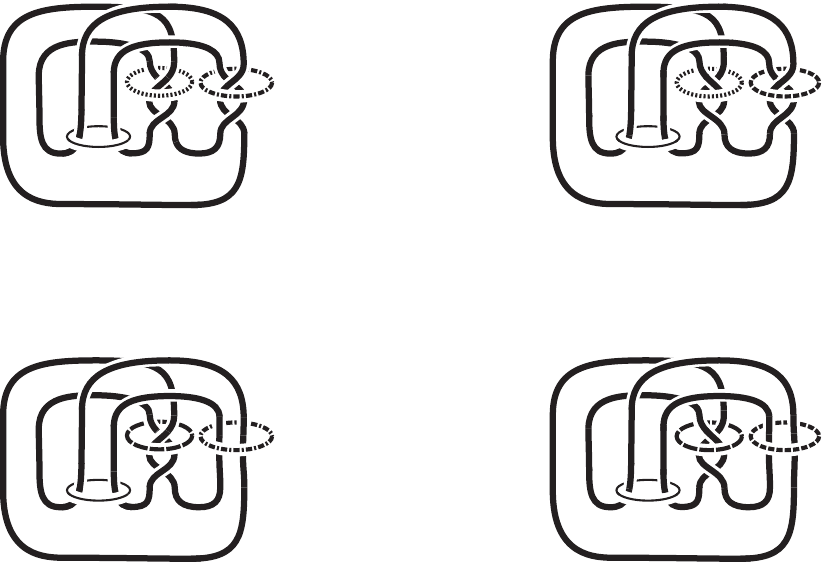}

\begin{picture}(400,0)(0,0)
\put(95,145){\scriptsize $2$}
\put(252,145){\scriptsize $2$}
\put(95,42){\scriptsize $2$}
\put(252,42){\scriptsize $2$}
\put(165,145){order $2$}
\put(165,55){order $2$}
\put(170,140){$\longleftrightarrow$}
\put(90,90){classical\ $\updownarrow$}
\put(250,90){classical\ $\updownarrow$}
\put(170,45){$\longleftrightarrow$}
\put(30,180){$GOF(0;1,2)\cong GOF(0;2,1)$}
\put(200,180){$GOF(0;2,-1)\cong GOF(0;-1,2)$}
\put(30,0){$GOF(0;1,-2)\cong GOF(0;-2,1)$}
\put(200,0){$GOF(0;-2,-1)\cong GOF(0;-1,-2)$}
\end{picture}
\caption{Generalized crossing changes between GOF-knots in $L(2,1)$.}
\label{fig:crossingchangeL(2,1)}
\end{center}
\end{figure}

We provide precise definitions of relevant terms in Section \ref{section:Definitions}. In Section \ref{section:CrossingChanges} we will discuss monodromies giving rise to clean or once-unclean arcs and 
we will identify the manifolds in which the relevant GOF-knots sit, and classify generalized crossing changes between GOF-knots, proving Corollary \ref{cor:ManifoldsForGOFKnotsWithGeneralizedCrossingChanges} and Theorem \ref{thm:CrossingChanges}. In Section \ref{section:ClassesCleanOnceUncleanArcs}, we prove Corollary \ref{cor:GOFequivalence}. In the Appendix, we provide additional surgery descriptions to show that certain manifolds are homeomorphic.

%


\section{Definitions and Background}
\label{section:Definitions}


\subsection{Automorphisms}

Let $F$ be a compact, connected, oriented surface with boundary. Suppose $\alpha$ is an arc properly embedded in the surface $F$,  and $h$ is a homeomorphism $h : F \to F$ so that the restriction of $h$ to the boundary is the identity. As $h$ fixes the boundary pointwise, $\alpha$ and $h(\alpha)$ necessarily share their endpoints. For this reason, whenever we say that two arcs $\alpha$ and $\beta$ properly embedded in a surface $F$ are \emph{disjoint}, we shall mean that they are disjoint on their interiors. 

Thus, an arc $\alpha$ is said to be \emph{clean} (with respect to $h$) if $\alpha$ and $h(\alpha)$ are disjoint, (i.e., $int(\alpha) \cap int(h(\alpha)) = \emptyset$). We will also say that $\alpha$ is \emph{once-unclean} (with respect to $h$) if $|int(\alpha) \cap int(h(\alpha))| = 1$.

Assume that $\alpha$ and $h(\alpha)$ have been isotoped (fixing endpoints) to intersect minimally. In general, $\alpha \cup h(\alpha)$ will be a curve in $F$ with self-intersections. We may move the endpoints $\bd \alpha = \bd h(\alpha)$ slightly into the interior of $F$ to obtain a curve immersed in the interior of $F$. Choose an orientation on $\alpha$. There is an induced orientation on $h(\alpha)$ so that $\alpha \cup h(\alpha)$ has a coherent orientation that agrees with the orientation of $\alpha$. Then the intial point of $\alpha$ is the terminal point of $h(\alpha)$ and vice versa. Say that $\alpha$ is \emph{right-veering} if the orientations induced by the tangent vectors to $h(\alpha)$ then $\alpha$ are opposite the orientation on $F$ at both endpoints of the arcs. We say that $\alpha$ is \emph{left-veering} if these orientations agree with the orientation on $F$ at both endpoints of the arcs. In either case, we say that the arc $\alpha$ is \emph{alternating}, as $h(\alpha)$ approaches $\alpha$ on alternate sides at the endpoints. Otherwise, we say that $\alpha$ is \emph{non-alternating}. See Fig. \ref{figure:OrientationsAlternatingNonAlternating}.

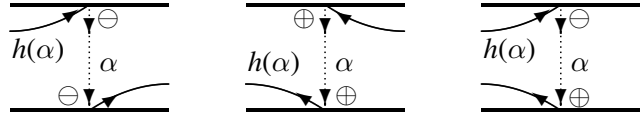
\begin{figure}[h]
\begin{center}
\begin{minipage}{.23\textwidth}
\begin{center}
\begin{picture}(80,40)(0,0)
\put(33,15){$\alpha$}
\linethickness{0.2mm}
\qbezier[20](30,0)(30,20)(30,40)
\put(0,20){$h(\alpha)$}
\qbezier(30,0)(45,10)(60,10)
\qbezier(30,40)(15,30)(0,30)
\linethickness{0.5mm}
\put(0,0){\line(1,0){60}}
\put(0,40){\line(1,0){60}}
\thicklines
\put(30,3){\vector(0,-1){1}}
\put(30, 30){\vector(0,-1){1}}
\put(40,4){\vector(4,2){1}}
\put(17, 3){$\ominus$}
\put(25,36){\vector(4,2){1}}
\put(32, 32){$\ominus$}
\end{picture}
\end{center}
\end{minipage}
\begin{minipage}{.23\textwidth}
\begin{center}
\begin{picture}(80,40)(0,0)
\put(33,15){$\alpha$}
\linethickness{0.2mm}
\qbezier[20](30,0)(30,20)(30,40)
\put(0,15){$h(\alpha)$}
\qbezier(30,0)(15,10)(0,10)
\qbezier(30,40)(45,30)(60,30)
\linethickness{0.5mm}
\put(0,0){\line(1,0){60}}
\put(0,40){\line(1,0){60}}
\thicklines
\put(20,6){\vector(-3,2){1}}
\put(35,37){\vector(-3,2){1}}
\put(30, 3){\vector(0,-1){1}}
\put(30, 30){\vector(0,-1){1}}
\put(33, 3){$\oplus$}
\put(17, 32){$\oplus$}
\end{picture}
\end{center}
\end{minipage}
\begin{minipage}{.23\textwidth}
\begin{center}
\begin{picture}(80,40)(0,0)
\put(33,15){$\alpha$}
\linethickness{0.2mm}
\qbezier[20](30,0)(30,20)(30,40)
\put(0,20){$h(\alpha)$}
\qbezier(30,0)(15,10)(0,10)
\qbezier(30,40)(15,30)(0,30)
\linethickness{0.5mm}
\put(0,0){\line(1,0){60}}
\put(0,40){\line(1,0){60}}
\thicklines
\put(30,3){\vector(0,-1){1}}
\put(30, 30){\vector(0,-1){1}}
\put(20,6){\vector(-3,2){1}}
\put(32, 2){$\oplus$}
\put(25,36){\vector(4,2){1}}
\put(32, 32){$\ominus$}
\end{picture}
\end{center}
\end{minipage}

\caption{The orientation induced by the tangent vectors to $h(\alpha)$ and then $\alpha$ either disagree with the orientation of $F$ at both endpoints (alternating, right-veering), agree with the orientation of $F$ at both endpoints (alternating, left-veering), or agree at one and disagree at the other endpoint (non-alternating).}
\label{figure:OrientationsAlternatingNonAlternating}
\end{center}
\end{figure}

Further, we will refer to a self-intersection point of $\alpha \cup h(\alpha)$, as a \emph{crossing}. We say that the crossing is \emph{positive} if the orientation induced by the tangent vectors to $h(\alpha)$ and then $\alpha$ agrees with the orientation on $F$, and \emph{negative} if this orientation disagrees with that of $F$. See Fig. \ref{figure:OreintationsCrossings}.

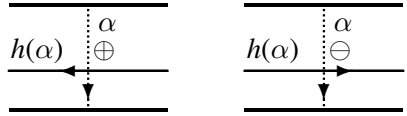
\begin{figure}[h]
\begin{center}
\begin{minipage}{.23\textwidth}
\begin{center}
\begin{picture}(80,40)(0,0)
\put(33,30){$\alpha$}
\linethickness{0.3mm}
\qbezier[20](30,0)(30,20)(30,40)
\put(0,20){$h(\alpha)$}
\qbezier(0,15)(30,15)(60,15)
\linethickness{0.5mm}
\put(0,0){\line(1,0){60}}
\put(0,40){\line(1,0){60}}
\thicklines
\put(20,15){\vector(-1,0){1}}
\put(30, 5){\vector(0,-1){1}}
\put(31, 20){$\oplus$}
\end{picture}
\end{center}
\end{minipage}
\begin{minipage}{.23\textwidth}
\begin{center}
\begin{picture}(80,40)(0,0)
\put(33,30){$\alpha$}
\linethickness{0.3mm}
\qbezier[20](30,0)(30,20)(30,40)
\put(0,20){$h(\alpha)$}
\qbezier(0,15)(30,15)(60,15)
\linethickness{0.5mm}
\put(0,0){\line(1,0){60}}
\put(0,40){\line(1,0){60}}
\thicklines
\put(40,15){\vector(1,0){1}}
\put(30, 5){\vector(0,-1){1}}
\put(31, 20){$\ominus$}
\end{picture}
\end{center}
\end{minipage}
\caption{A crossing is positive or negative depending on whether the orientation induced by tangent vectors to $h(\alpha)$ and then to $\alpha$ agree or disagree with the orientation on $F$, respectively.}
\label{figure:OreintationsCrossings}
\end{center}
\end{figure}

 \subsection{Surface bundles, open book decompositions, and monodromy maps}

Let $I$ be the unit interval $[0, 1]$. Given a homeomorphism $h: F \to F$ as above, we can form $(F \times I)/\sim$, where $(x, 1) \sim (h(x), 0)$ for all $x \in F$, the \emph{surface bundle over $S^1$}. The map $h$ is called the \emph{monodromy} of the bundle, and the bundle can also be denoted $(F \times I)/h$. Each copy of $F$ arising from $F \times \set{y}$ is called a \emph{fiber}. The resulting manifold is well-defined up to conjugation of $h$ in the mapping class group of $F$, and Dehn-twisting along curves in $F$ parallel to boundary components of $F$.

The surface bundle formed above has a toroidal boundary component arising from each boundary component of $F$. If we fill each toral boundary component with a solid torus so that each loop in the torus arising from $(\set{x} \times I) / h$ bounds a disk in the solid torus, where $x \in \bd F$, the result is a closed $3$-manifold, $M$. The union of the cores of all so-filled solid tori forms a link in this $3$-manifold. This link is often referred to as a \emph{fibered link} in $M$. In this language, each copy of the surface $F$ is again called a \emph{fiber}. Alternatively, the link is called the \emph{binding} of an \emph{open book decomposition} of $M$. In this language, each copy of the surface $F$ is called a \emph{page}. For the purposes of this paper, we will largely use the terms interchangably, often preferring the language of fibrations or surface bundles for ease of exposition.

Given a particular page $F_0$ in an open book decomposition, and an arc $\alpha$ properly embedded in $F_0$, let $n(\alpha)$ denote a regular neighborhood of $\alpha$ in $M$. Then there is a unique loop $L$ in $\bd n(\alpha) \rmv F_0$ that bounds a disk in the manifold intersecting the page $F_0$ in exactly the arc $\alpha$. We will call $L$ an \emph{$\alpha$-loop for the page $F_0$}.

\subsection{The arc complex and isometric actions on $\mathbb{H}^2$}
\label{subsection:arccomplex}
The \emph{arc complex} $\AC(F)$ of a surface $F$ with boundary is a simplicial complex whose vertices correspond to the proper isotopy classes of (essential) arcs properly embedded in $F$, and whose vertices span a simplex if the vertices correspond to isotopy classes of arcs that can be made pairwise disjoint (on their interiors) in $F$. 

Suppose $F$ is a once-punctured torus. In this case, $\AC(F)$ is two-dimensional. In fact, by shrinking the boundary of $F$, isotopy classes of essential arcs in $F$ are in one-to-one correspondence with essential simple closed curves in the torus, which, in turn, are in one-to-one correspondence with $\mathbb{Q} \cup \{\infty\}$, the set of slopes on the torus. Further, two arcs in the punctured torus $F$ can be istoped to intersect minimally in $n$ points (in their interiors) if and only if their corresponding ratios $p/q$ and $p'/q'$ (in lowest terms, or $\infty = 1/0$) satisfy: $|pq' - qp'| = n+1$ (see, for instance, \cite{HatThuIS2BKC}).

It is well known that the $1$-skeleton of the arc complex of a once-punctured torus is the \emph{Farey graph}, and that the complex $\AC(F)$ has a very useful embedding into $\overline{\mathbb{H}^2}$, the Gromov compactification of the hyperbolic plane. Each $2$-dimensional simplex embeds as an ideal triangle, and each $1$-simplex embeds as a geodesic line. There is also an associated dual tree $\T$, which embeds in $\mathbb{H}^2$ by taking a vertex at the orthocenter of each triangle of $\AC(F)$, and joining two vertices arising from triangles in $\AC(F)$ sharing an edge. (See \cite{FloHatISPTB}.)

Suppose two essential arcs $\alpha$ and $\beta$ properly embedded in $F$ are disjoint and non-isotopic.   
They correspond to two vertices spanning a $1$-simplex of $\AC(F)$, and cutting  $F$ open along $\alpha\cup \beta$ results in a disk.
Then by Alexander's Trick, an automorphism $h$ of $F$ is determined up to free isotopy by the images of $\alpha$ and $\beta$ under $h$.

Further, an orientation-preserving homeomorphism of $F$, which can be identified with an element of $SL(2, \mathbb{Z})$, induces an automorphism of $\AC(F)$, an automorphism of $\T$, and an orientation-preserving isometry of $\mathbb{H}^2$ which extends to a continous map of $\overline{\mathbb{H}^2}$, agreeing with the actions on $\AC(F)$ and $\T$.

So, in particular, the monodromy map $h: F \to F$ induces an isometry $\widetilde{h}: \mathbb{H}^2 \to \mathbb{H}^2$. By a slight (and common) abuse of notation, we will refer to both the isometry on $\mathbb{H}^2$ and its extension to $\overline{\mathbb{H}^2}$ by $\widetilde{h}$. By the classification of hyperbolic isometries, $\widetilde{h}$ is one of three classes: (1) elliptic, (2) parabolic, or (3) loxidromic, which correspond exactly to $h$ being (1) periodic, (2) reducible, or (3) pseudo-Anosov  \cite{ThuOGDDS}.

Another observation clarifies an important distinction between automorphisms of the once-puctured torus and the induced automorphisms on the arc complex for the once-punctured torus. The vertices of the arc complex $\AC(F)$ are unoriented arcs. The oriented arc complex $\OAC(F)$ has vertices represented by oriented arcs in $F$. Then $\OAC(F)$ double-covers $\AC(F)$, and there is a short exact sequence relating the automorphism groups,
\[ \set{1} \to \mathbb{Z}_2 \to Aut(\OAC(F)) \stackrel{\pi}{\to} Aut(\AC(F)) \to \set{1} .\]

The hyper-elliptic involution, $\tau$, is a non-trivial automorphism, even up to free isotopy, but preserves all free isotopy classes of arcs set-wise, reversing their orientations. Thus, the induced action of $\tau$ on $\OAC(F)$ generates the kernel of $\pi$. This involution does not fix the boundary of $F$, so it is not a monodromy map.


\section{Crossing Changes}
\label{section:CrossingChanges}

In this section, we will characterize all generalized crossing changes between two GOF-knots. We will first establish a correspondence between such generalized crossing changes and clean or once-unclean arcs in a fiber. Then we will classify all monodromies giving rise to such arcs. We will next describe the ambient spaces in which the fibered knots with these monodromies lie. Finally, we will classify all generalized crossing changes between GOF-knots.

\subsection{Clean and once-unclean arcs}
\label{subsection:CleanOnceUncleanArcs}
Baker, Johnson, and Klodginski classify once-punctured torus bundles that have tunnel number one, showing that they must be knot complements in lens spaces, $L(r, 1)$ \cite{BakJohKloTN1G1FK}. This is closely related to once-punctured torus bundles with a clean arc with respect to the monodromy. Coward and Lackenby (\cite{CowLacUGOK}) have shown that if a GOF-knot has a clean arc that is alternating, then there are at most two  distinct such arcs, up to monodromy equivalence. In \cite{BucIshRatShiBSCCBFL}, the authors with Buck and Shimokawa, investigate the related class of once-unclean arcs in fiber surfaces, and give a geometric characterization of when such arcs arise.

We recall that a \emph{crossing circle} for a knot (or link) $K$ is a circle $L$ that bounds a disk intersecting $K$ in two points with opposite orientations. We refer to the disk as a \emph{crossing disk}. Then, a \emph{generalized crossing change along $L$ of order $q$} is a $- \frac{1}{q}$ Dehn surgery on $L$, with $q \in \mathbb{Z} \rmv \set{0}$. Since $L$ bounds a disk, the ambient manifold does not change, but the knot may. When $q = \pm 1$, this is just an ordinary \emph{crossing change}. Also, $\chi(K)$ refers to the maximal Euler characteristic of all Seifert surfaces for $K$, and a Seifert surface $S$ for $K$ is said to be \emph{taut} if its Euler characteristic realizes $\chi(K)$.

\begin{lem}
\label{lem:AllCrossingChangesAlongArcs}
If $K$ is a GOF-knot with fiber $F$, $L$ is a crossing circle for $K$, and the result of an order $q$ generalized crossing change ($q$-twist) along $L$ is another GOF-knot, then $L$ bounds a disk that intersects $F$ in a single arc $\alpha$. Moreover, one of the following holds:
\begin{enumerate}
\item $q=\pm 2$, $\alpha$ is clean and alternating (not fixed) with respect to the monodromy of $F$, or 
\item $q=\pm 1$, $\alpha$ is once-unclean (and alternating) with respect to the monodromy of $F$. 
\end{enumerate}
\end{lem}

\begin{proof}
Our method is similar to the proofs in $S^3$ from \cite{KalLinKAGET} and \cite{SchThoLGCM}, relying on an important result of Gabai in \cite{GabFT3MII}. Evidently, $\chi(K) = -1$. Suppose that $S$ is a taut surface bounded by $K$ in the complement of $L$. From the local picture, the crossing disk must intersect $S$ in a single arc. Let $K'$ and $S'$ be the images of $K$ and $S$, respectively, after the generalized crossing change, and note that $\chi(S') = \chi(S)$. By Corollary 2.4 of \cite{GabFT3MII}, at least one of $S$ or $S'$ is taut for $K$ or $K'$. But then they both realize $\chi(K) = \chi(K')$, so, in particular, $S$ must be taut for $K$. 
There are no Euler characteristic $-1$ surfaces in a once-punctured torus bundle besides the fiber, so $S=F$, and the first part of the statement is established.

Now, in exactly the same way as obtaining Theorem 5 from Theorem 3 in \cite{BucIshRatShiBSCCBFL}, 
we have that if a crossing disk intersects a fiber surface in an arc, and the result of the generalized crossing change is another fiber bundle, then one of the two cases in the statement of the lemma 
occurs, or the arc is clean and non-alternating.  
However, Lemma \ref{lemma:Alternating} excludes the latter possibility.
\end{proof}

The converse of the first statement also holds -- for a given arc $\alpha$ properly embedded in a fiber surface, there is a uniquely determined crossing circle $L$ so that $L$ bounds a disk intersecting the fiber surface in $\alpha$.
We call such a crossing circle an \emph{$\alpha$-loop}, denoted by $L_{\alpha}$.
Lemma \ref{lem:AllCrossingChangesAlongArcs} implies that 
non-classical generalized crossing changes (resp., classical crossing changes) between GOF-knots must occur at $\alpha$-loops for arcs $\alpha$ that are clean and alternating (resp., once-unclean and alternating). 
Hence, it suffices to look at clean or once-unclean (and alternating) arcs. 

\subsection{Monodromies}
\label{subsection:Monodromies}

\begin{lem}
\label{lemma:Alternating}
Suppose $h$ is an automorphism of a once-punctured torus that is the identity on the boundary. Let $\alpha$ be any essential arc in the surface. Then $\alpha$ is non-alternating if and only if $h|_\alpha = Id_\alpha$. 
\end{lem}


\begin{proof}

It is clear that an arc fixed pointwise by $h$ is non-alternating. 
Suppose, then, that $h(\alpha)$ is not isotopic fixing endpoints to $\alpha$, but has been isotoped fixing endpoints to intersect $\alpha$ minimally. Then cutting the surface along $\alpha$ cuts $h(\alpha)$ into disjoint arcs, $\eta_1, \dots, \eta_k$ properly embedded in an annulus with endpoints contained in the two sub-arcs of the boundary corresponding to $\alpha$, say $\alpha^{\pm}$. Observe that none of the $\eta_1, \dots, \eta_k$ are parallel into either of $\alpha^{\pm}$, because $\alpha$ and $h(\alpha)$ intersected minimally.

Every intersection between $\alpha$ and $h(\alpha)$ in the interior of these arcs corresponds to one endpoint of $\bigcup \eta_i$ on each of $\alpha^{\pm}$. Call an arc $\eta_i$ a $(++)$-arc, $(+-)$-arc or $(--)$-arc, depending on the locations of the endpoints of $\eta_i$. Suppose that $\alpha$ is non-alternating. Then, without loss of generality, the number of $(+)$-endpoints is exactly two greater than the number of $(-)$-endpoints. Every essential arc in the annulus will be a $(+-)$-arc, so there must be at least one $(++)$-arc. 

Consider, now, $\eta_1$ and $\eta_k$, the arcs incident to each of the endpoints of $\alpha^+$. (Of course, $\eta_k \neq \eta_1$, or else $h(\alpha) = \eta_1$ is an inessential arc in both the annulus and the punctured torus, while $\alpha$ was essential in the punctured torus.) If $\eta_1$ were a $(+-)$-arc, then it would be essential in the annulus, and there could be no $(++)$-arcs, for they would have to be parallel into $\alpha^+$, which is impossible. On the other hand, if $\eta_1$ were a $(++)$-arc, then it would be inessential in the annulus and would have one endpoint at an endpoint of $\alpha^+$ and the other endpoint in the interior of $\alpha^+$. Then $\eta_1$ would separate the annulus, and $\eta_k$ would be parallel into $\alpha^+$, which is still impossible.
\end{proof} 

Now, suppose $F$ is an oriented once-punctured torus.   
For any essential arc $\alpha$ properly embedded in $F$, 
there is a uniquely determined (up to isotopy) essential loop $c_\alpha$ in $F$ 
disjoint from $\alpha$. Let $D_\alpha$ denote the right-handed Dehn twist along the curve $c_\alpha$. 
Take a pair of disjoint essential arcs $\alpha$ and $\beta$ in $F$ that together cut $F$ into a disk. 
Any orientation preserving self-homeomorphism of $F$ can be represented by a composition of $D_\alpha$ and $D_\beta$.
In particular, $(D_\alpha \circ D_\beta)^{3}$ represents a unique homeomorphism regardless of the choice of $\alpha$ and $\beta$. 
It is freely isotopic to the hyper-elliptic involution, 
We denote $(D_\alpha \circ D_\beta)^3$ by $(D_{\bd})^{1/2}$, since 
$\left((D_{\bd})^{1/2}\right)^2$ is isotopic to a single Dehn twist around a curve parallel to the boundary of $F$.
Any two \emph{monodromy} maps that are freely isotopic differ by some power of $(D_{\bd})^{1/2}$. 
Then we have the following two lemmas by using this representation.


\begin{lem}
\label{lemma:FactorizationClean}
If $\kappa$ is a clean arc with respect to an orientation preserving self-homeomorphism $h$ of a once-punctured torus $F$, 
then $h = D_{\delta}^{\pm1} \circ D_{\kappa}^n$ for some arc $\delta$ in $F$ disjoint from $\kappa$ and some integer $n$.
\end{lem}
\begin{proof}
Suppose $\kappa$ is an arc in a once-punctured torus that is clean with respect to $h$. Observe that the punctured torus cut along $\kappa$ is an annulus. Thus, any homeomorphism that actually fixes $\kappa$ must be $D_\kappa^n$ for some $n \in \mathbb{Z}$.

So, if $\kappa$ is fixed, then let $\delta = \kappa$, and the result holds. 

Otherwise, by Lemma \ref{lemma:Alternating}, $\kappa$ is an alternating arc. Since $\kappa$ is clean and alternating, $c = \kappa \cup h(\kappa)$ is a simple closed curve (we may take it to be in the interior of the surface by moving the two points $\bd \kappa = \bd h(\kappa)$ slightly into the interior). There is a unique arc, $\delta$, in the surface distinct from $\kappa$ and disjoint from $c$, and $c = c_\delta$. Now $(D_{\delta})^{\mp1}
 \circ h$ fixes $\kappa$, so $(D_{\delta})^{\mp1}
  \circ h$ is equal to $D_{\kappa}^n$ for some $n \in \mathbb{Z}$.   \end{proof}

When an orientation preserving automorphism, $h$, has a once-unclean arc, $\mu$, recall that $\mu \cup h(\mu)$ can be considered an immersed curve with a single crossing, and that orienting $\mu$ (either way) gives a well-defined sign to the crossing because of an induced orientation on $h(\mu)$. There are then two ways of resolving this crossing. The resolution that is consistent with the orientations results in two simple closed curves. Call these curves $r$ and $s$. The other resolution results in a single simple closed curve. Call this curve $t$.

\begin{lem}\label{lemma:FactorizationUnclean}%
If $\mu$ is a once-unclean arc with respect to  an orientation preserving self-homeomorphism $h$ of a once-punctured torus $F$,  then one of the following holds:
\begin{enumerate} 
\item \label{case:RVNC} The arc $\mu$ is either right-veering with a negative crossing, or is left-veering with a positive crossing, $t$ is trivial, $r = s$ are both essential loops isotopic 
to $c_\delta$ for some essential arc $\delta$ disjoint from $\mu$, and $h = D_{\delta}^{\pm2
} \circ D_{\mu}^n$ for some $n \in \mathbb{Z}$.

\item \label{case:RVPC} The arc $\mu$ is either right-veering with a positive crossing, or is left-veering with a negative crossing, $t$ is parallel to the boundary of $F$, $r = s$ are both essential loops disjoint from $\mu$ and $h(\mu)$, and $h = (D_\bd)^{\pm1/2} \circ D_\mu^n$ for some $n \in \mathbb{Z}$.
\end{enumerate}
\end{lem}

\begin{proof}

First suppose that $\mu$ is right-veering with a negative crossing or left-veering with a positive crossing. See Fig. \ref{figure:RVNegCrossingLVPosCrossing}.

\begin{figure}[h]
\begin{center}
\begin{minipage}{.23\textwidth}
\begin{center}
\begin{picture}(80,40)(0,0)
\put(33,25){$\mu$}
\linethickness{0.3mm}
\qbezier[20](30,0)(30,20)(30,40)
\put(0,20){$h(\mu)$}
\qbezier(30,0)(45,10)(60,10)
\qbezier(30,40)(15,30)(0,30)
\qbezier(0,15)(30,15)(60,15)
\linethickness{0.5mm}
\put(0,0){\line(1,0){60}}
\put(0,40){\line(1,0){60}}
\thicklines
\put(50,15){\vector(1,0){1}}
\put(55,10){\vector(4,1){1}}
\put(10,31){\vector(4,1){1}}
\end{picture}
\end{center}
\end{minipage}
\begin{minipage}{.23\textwidth}
\begin{center}
\begin{picture}(80,40)(0,0)
\put(33,25){$\mu$}
\linethickness{0.3mm}
\qbezier[20](30,0)(30,20)(30,40)
\put(0,20){$h(\mu)$}
\qbezier(30,0)(15,10)(0,10)
\qbezier(30,40)(45,30)(60,30)
\qbezier(0,15)(30,15)(60,15)
\linethickness{0.5mm}
\put(0,0){\line(1,0){60}}
\put(0,40){\line(1,0){60}}
\thicklines
\put(10,15){\vector(-1,0){1}}
\put(5,10){\vector(-4,1){1}}
\put(50,31){\vector(-4,1){1}}
\end{picture}
\end{center}
\end{minipage}
\caption{Right-veering with a negative crossing or left-veering with a positive crossing.}
\label{figure:RVNegCrossingLVPosCrossing}
\end{center}
\end{figure}
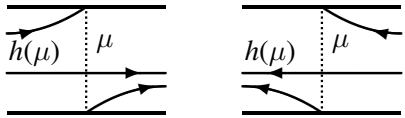

Since both $r$ and $s$ intersect $\mu$ exactly once, they must be essential, but $r$ and $s$ are also disjoint, so they must be isotopic. So, $r$ and $s$ are equal to $c_\delta$ for some essential arc disjoint from $\mu$. Now, as $r$ and $s$ are isotopic, they cobound an annulus in $F$, and the curve $t$ is the boundary of the disk obtained by cutting this annulus along an essential arc, so $t$ is trivial. Finally, observe that $h(\mu)$ is obtained from $\mu$ precisely by twisting twice positively (respectively, negatively) around $r = s = c_\delta$ when $\mu$ is right-veering (respectively, left-veering). So $D_{\delta}^{\mp2} \circ h$ is equal to $D_\mu^n$ for some $n \in \mathbb{Z}$.

Next, suppose that $\mu$ is right-veering with a positive crossing or left-veering with a negative crossing. See Fig. \ref{figure:RVPosCrossingLVNegCrossing}. 

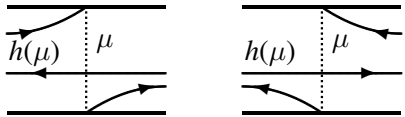
\begin{figure}[h]
\begin{center}
\begin{minipage}{.23\textwidth}
\begin{center}
\begin{picture}(80,40)(0,0)
\put(33,25){$\mu$}
\linethickness{0.3mm}
\qbezier[20](30,0)(30,20)(30,40)
\put(0,20){$h(\mu)$}
\qbezier(30,0)(45,10)(60,10)
\qbezier(30,40)(15,30)(0,30)
\qbezier(0,15)(30,15)(60,15)
\linethickness{0.5mm}
\put(0,0){\line(1,0){60}}
\put(0,40){\line(1,0){60}}
\thicklines
\put(10,15){\vector(-1,0){1}}
\put(55,10){\vector(4,1){1}}
\put(10,31){\vector(4,1){1}}
\end{picture}
\end{center}
\end{minipage}
\begin{minipage}{.23\textwidth}
\begin{center}
\begin{picture}(80,40)(0,0)
\put(33,25){$\mu$}
\linethickness{0.3mm}
\qbezier[20](30,0)(30,20)(30,40)
\put(0,20){$h(\mu)$}
\qbezier(30,0)(15,10)(0,10)
\qbezier(30,40)(45,30)(60,30)
\qbezier(0,15)(30,15)(60,15)
\linethickness{0.5mm}
\put(0,0){\line(1,0){60}}
\put(0,40){\line(1,0){60}}
\thicklines
\put(50,15){\vector(1,0){1}}
\put(5,10){\vector(-4,1){1}}
\put(50,31){\vector(-4,1){1}}
\end{picture}
\end{center}
\end{minipage}
\caption{Right-veering with a positive crossing or left-veering with a negative crossing.}
\label{figure:RVPosCrossingLVNegCrossing}
\end{center}
\end{figure}

It is clear from the Fig. \ref{figure:RVPosCrossingLVNegCrossing} that $r$ and $s$ are both disjoint from $\mu$ and $h(\mu)$, and neither is boundary parallel. If either of them were trivial, then $h(\mu)$ could be isotoped fixing endpoints across the disk bounded by the curve to intersect $\mu$ fewer times, and then $\mu$ would be clean instead of once-unclean. As there cannot be disjoint essential curves in a once-punctured torus, $r$ is isotopic to $s$.

Since $r$ and $s$ are non-trivial, the geometric intersection number between $t$ and $\mu$ is exactly two, so $t$ cannot be trivial. But, on a once-punctured torus, the geometric intersection number of any essential simple closed curve with an essential arc is the absolute value of the algebraic intersection number, while the algebraic intersection number between $t$ and $\mu$ is zero, so $t$ cannot be an essential curve. Thus, $t$ must be nontrivial but inessential, so it is boundary parallel. 
 
Finally, since $t$ is boundary parallel, but traces the path of $h(\mu)$, observe that the half-twist around the boundary, $(D_\bd)^{-1/2}$ (respectively, $(D_\bd)^{1/2}$), carries $h(\mu)$ to $\mu$ when $\mu$ is right-veering with a positive crossing (respectively, left-veering with a negative crossing). Hence $(D_\bd)^{\mp1/2} \circ h$ fixes $\mu$, so is equal to $D_{\mu}^n$ for some $n \in \mathbb{Z}$.
\end{proof}

The following lemma describes how the monodromy is changed by a generalized crossing change between GOF-knots. Recall that the monodromy of a fibered knot is defined up to conjugation.

\begin{lem}\label{lem:MonodromyChange} Assume the hypotheses of Lemma \ref{lem:AllCrossingChangesAlongArcs}.
\begin{enumerate}
\item If $\kappa$ is a clean and alternating arc of intersection between the crossing disk and the fiber, then an order $\pm2$ generalized crossing change along the $\kappa$-loop changes the monodromy by composition with $D_{\delta}^{\pm2}$.
$$D_\delta \circ D_\kappa^n\quad\autorightleftharpoons{$D_{\delta}^{-2}$}{$D_{\delta}^2$}\quad D_\delta^{-1} \circ D_\kappa^n.$$
\item[{\rm(2-1)}] If $\mu$ is a once-unclean arc of intersection between the crossing disk and the fiber, and there is an essential arc $\delta$ disjoint from $\mu$, then an order $\pm1$ classical crossing change along the $\mu$-loop 
changes the monodromy by composition with $D_{\delta}^{\pm4}$.
$$D_\delta^2 \circ D_\mu^n\quad\autorightleftharpoons{$D_{\delta}^{-4}$}{$D_{\delta}^4$}\quad D_\delta^{-2} \circ D_\mu^n.$$
\item[{\rm(2-2)}] If $\mu$ is a once-unclean arc of intersection between the crossing disk and the fiber, and there is no essential arc disjoint from $\mu$, then an order $\pm1$ classical crossing change along the $\mu$-loop 
changes the monodromy by composition with $D_{\mu}^{\pm4}\circ (D_\bd)^{\mp1}$.
$$(D_\bd)^{1/2} \circ D_\mu^n\quad\autorightleftharpoons{$D_{\mu}^{4}\circ (D_\bd)^{-1}$}{$D_{\mu}^{-4}\circ (D_\bd)$}\quad (D_\bd)^{-1/2} \circ D_\mu^{n+4}.$$
\end{enumerate}
\end{lem}
\begin{proof}
\begin{enumerate} 
\item If $\kappa$ is a clean and alternating arc, then by Lemma \ref{lemma:FactorizationClean} the monodromy $h$ has the form $D_\delta^{\pm1} \circ D_\kappa^n$. 
The $\kappa$-loop $L_\kappa$ is isotopic to $c_\delta$ in the surface bundle $(F\times I)/h$. 
Then the conclusion (1) holds by Lemma \ref{lem:AllCrossingChangesAlongArcs} and Proposition 1.4 of \cite{NiDSKPM}.
\item If $\mu$ is a once-unclean arc, then the conclusions of (2-1) or (2-2) follow from Lemma \ref{lemma:FactorizationUnclean} and  Corollary 5 of \cite{BucIshRatShiBSCCBFL} which is based on Proposition 1.4 of \cite{NiDSKPM}. \qedhere
\end{enumerate}
\end{proof}

Recall that an orientation preserving automorphism $h : F\to F$ induces an isometry $\widetilde{h}:\mathbb{H}^2\to\mathbb{H}^2$, that can be identified with an element of $SL(2,\mathbb{Z})$.
\begin{lem}\label{lem:matrix}
Let $\delta$, $\kappa$, and $\mu$ be essential arcs properly embedded in $F$. Suppose $\delta$ and $\kappa$ (resp., $\delta$ and $\mu$) are disjoint and not isotopic.
Then automorphisms 
\begin{center}
$D_{\delta}^{\pm1} \circ\,  D_{\kappa}^{n}$,  
$D_{\delta}^{\pm2} \circ\,  D_{\mu}^{n}$, and $(D_{\bd})^{\pm1/2}\circ\, D_{\kappa}^{n}$ 
\end{center}
can be represented by 
\begin{center}
$\begin{pmatrix} 1 & n \\ 
\mp1 & 1 \mp n 
\end{pmatrix}$, 
$\begin{pmatrix} 1 & n \\ 
\mp2 & 1 \mp 2n 
\end{pmatrix}$,
and 
$\begin{pmatrix} -1 & n \\ 
0 & -1 
\end{pmatrix}$,
respectively.
\end{center}
\end{lem}
\begin{proof}
Let $\alpha$ and $\beta$ be disjoint and non-isotopic essential arcs in $F$. 
Since the first homology of $F$ generated by  $[c_{\alpha}]$ and $[c_{\beta}]$, we put $[c_{\alpha}]=^t\!\!(1,0)$ and $[c_{\beta}]=^t\!\!(0,1)$.
Then the automorphisms $D_{\alpha}$ and $D_{\beta}$ are represented by  
either
$\begin{pmatrix} 
1 & 1 \\ 
0 & 1 
\end{pmatrix}$ and 
$\begin{pmatrix} 
1 & 0 \\ 
-1 & 1 
\end{pmatrix}$, or
$\begin{pmatrix} 
1 & -1 \\ 
0 & 1 
\end{pmatrix}$ 
and 
$\begin{pmatrix} 
1 & 0 \\ 
1 & 1 
\end{pmatrix}$,
according to the choice of orientations of 
$c_{\alpha}$ and $c_{\beta}$. 
Applying this for the arcs $\delta$ and $\kappa$ (resp.,  $\delta$ and $\mu$) we have the conclusion.
\end{proof}

Using Lemmas \ref{lemma:FactorizationClean}, \ref{lemma:FactorizationUnclean}, and \ref{lem:matrix}, together with the classification of orientation-preserving automorphisms of the once-punctured torus by the trace of the induced element of $SL(2, \mathbb{Z})$, we immediately have the following corollary.

\begin{cor}
\label{cor:ClassificationMonodromies}
Every orientation preserving automorphism of the once-punctured torus that fixes the boundary pointwise and admits an arc that is either clean or once-unclean appears in Table \ref{table:ClassificationMonodromies}, up to inverses and conjugation.
\end{cor}

\begin{table}[h!]
  \begin{center}
    \begin{tabular}{c|c}
      \textbf{Classification} & \textbf{Monodromy} \\
      \hline
      Periodic
      & $D_{\delta} \circ D_{\kappa}$\qquad
      $D_{\delta} \circ D_{\kappa}^2$ \qquad
      $D_{\delta} \circ D_{\kappa}^3$ \qquad
      $D_{\delta}^2 \circ D_{\mu}$ \\
      \hline
      \multirow{2}{*}{Reducible} & 
      \quad$D_{\kappa}^n$ \qquad
      $D_{\delta} \circ D_{\kappa}^0$\qquad 
      $D_{\delta} \circ D_{\kappa}^4$\qquad
      $D_{\delta}^2 \circ D_{\mu}^0$\qquad \\
      & \quad  $D_{\delta}^2 \circ D_{\mu}^2$\qquad  \quad    
      $(D_{\bd})^{1/2} \circ D_{\mu}^{n}$ \\
      \hline
      \multirow{2}{*}{Pseudo-Anosov} & 
      $D_{\delta}^{-1} \circ D_{\kappa}^n$ ($n > 0$)\qquad   
     $D_{\delta} \circ D_{\kappa}^n$ ($n > 4$)\\
    &  $D_{\delta}^{-2} \circ D_{\mu}^n$ ($n > 0$) \qquad
      $D_{\delta}^2 \circ D_{\mu}^n$ $(n > 2)$ \\
      \hline
    \end{tabular}
        \caption{Monodromies of the once-punctured torus bundle that admit clean arcs ($\kappa$) or once-unclean arcs ($\mu$), up to inverses and conjugation.}
    \label{table:ClassificationMonodromies}
  \end{center}
\end{table}

\begin{rem}
\label{rem:Redundancy}
We note that there is redundancy in Table \ref{table:ClassificationMonodromies}. For instance, when $n=2$ in $D_\delta^{-1} \circ D_\kappa^n$, we may instead think of $\delta$ as a once-unclean arc $\mu$, relabel $\kappa$ as $\delta'$, and recognize this as the monodromy $D_\mu^{-1} \circ D_{\delta'}^2$, which is the inverse of $D_{\delta'}^{-2} \circ D_\mu^k$ for $k = 1$. This points to the facts that the monodromies can be factored in multiple ways, and that a single monodromy might have multiple $h$-equivalence classes of clean and/or once-unclean arcs. This final subtlety will be addressed in Section \ref{section:ClassesCleanOnceUncleanArcs}.
\end{rem}

\subsection{Manifolds}
\label{subsection:Manifolds}

We can use Lemmas \ref{lemma:FactorizationClean} and \ref{lemma:FactorizationUnclean} to give a link-surgery description of every GOF-knot with a clean arc or once-unclean arc, and describe the manifolds in which the GOF-knots sit.

\begin{thm} \leavevmode
\label{thm:Manifolds}

\begin{enumerate}
\item Every once-punctured torus bundle with a clean arc is the complement of a GOF-knot in $L(n, 1)$ for some $n \in\mathbb{Z}$. 
\item Every once-punctured torus bundle with a once-unclean arc is the complement of a GOF-knot in a prism manifold, $L(4,\pm1)$, or $L(2, 1)\, \sharp\, L(n,1)$ for some $n\in \mathbb{Z}$.
\end{enumerate}
\end{thm}

\begin{proof}

Let $GOF(0;k,\ell)$ and $GOF(\pm1;m)$ be GOF-knots with monodromy $D_\beta^\ell\circ D_\alpha^k$ and $(D_\bd)^{1/2} \circ D_\alpha^m$, respectively, and let $M(0;k,\ell)$ and $M(\pm1;m)$, respectively, be the manifolds in which they sit. 

We will describe the resulting knot complement as the result of a particular Dehn surgery on the trefoil knot in $S^3$. 
With disjoint arcs $\alpha$ and $\beta$ in a once-punctured torus $F$, the monodromy of the trefoil knot is represented by $D_\alpha\circ D_\beta$. 
Namely, the exterior of the trefoil knot is homeomorphic to the manifold which is obtained from $F\times [0,1]$ by identifying two points $(x,1)$ and $(h(x),0)$, and the meridian corresponds to $\set{y} \times [0,1]$ for a point $y$ in $\bd F$, see Fig. \ref{fig:TrefoilFiber}. 
For a loop $c$ in $F\times \{*\}$, we consider $\left(\frac{n\ell_c+1}{n}\right)$-surgery along $c$, where $\ell_c$ is the linking number of $c$ with a loop parallel to $c$ in $F\times \{*\}$.
This surgery corresponds to the operation of cutting the fiber bundle along $F\times \{*\}$ and gluing it again after twisting $n$-times along $c$. 
Then the resulting manifold is a new once-punctured torus bundle, whose monodromy is changed by $D_c^n$ from the original one, where $D_c$ is a Dehn twist along $c$.
We will use this method multiple times to give surgery descriptions of $GOF(0;k,\ell)$ in $M(0;k,\ell)$ and $GOF(\pm1;m)$ in $M(\pm1;m)$. 
Note that $\ell_c=1$ if $c=c_\alpha\times\{*\}$ or $c=c_\beta\times\{*\}$.
 
In the case of $GOF(0;k,\ell)$ in $M(0;k,\ell)$, which has a monodromy $h=D_\beta^\ell\circ D_\alpha^k$, we use two loops $c_1=c_\alpha\times\left\{\frac{1}{3}\right\}$ and $c_2=c_\beta\times\left\{\frac{2}{3}\right\}$ to provide a surgery description.
The surgery coefficient are $\frac{k}{k-1},\frac{\ell}{\ell-1}$ for $c_1,c_2$ respectively.
Then the resulting manifold is a once-punctured torus bundle with the monodromy $D_\beta^{\ell-1}\circ D_\alpha^{k}\circ D_\beta$, which is conjugate to $h$. 
Let $L$ be the $\alpha$-loop for the fiber $F \times \set{0}$.
By the Kirby calculus, we have a surgery description of $GOF(0;k,\ell)$, together with the $\alpha$-loop, $L$, see Fig. \ref{fig:surgerydescriptionGOF(0;k,l)}.
The manifold $M(0;k,\ell)$ is homeomorphic to $L(-\ell,1)\, \sharp\, L(-k,1)$. 
In particular, 
\begin{center}
$M(0;k,\ell)=\begin{cases}
L(-k,1)&(\ell=\pm1)\\
L(2,1)\, \sharp\, L(-k,1)&(\ell=\pm2).\\
\end{cases}$
\end{center}

In the case of $GOF(1;m)$ in $M(1;m)$, which has a monodromy $h=(D_\bd)^{1/2} \circ D_\alpha^m$, we use three loops $c_1=c_\alpha\times\left\{\frac{1}{4}\right\},c_2=c_\beta\times\left\{\frac{1}{2}\right\},c_3=c_\alpha\times\left\{\frac{3}{4}\right\}$ for a surgery description.
The surgery coefficients are $2,2,\frac{m+3}{m+2}$ for $c_1,c_2,c_3$ respectively.
Then the resulting manifold is a once-punctured torus bundle with the monodromy $D_\alpha^{m+2}\circ D_\beta\circ D_\alpha^2\circ D_\beta$, which is conjugate to $h=(D_\alpha\circ D_\beta)^3\circ D_\alpha^m$. 
Let $L$ be the $\alpha$-loop for the fiber $F \times \set{ \frac{3}{4} }$.
By the Kirby calculus, we have a surgery description of $GOF(1;m)$, together with the $\alpha$-loop, $L$, see Fig. \ref{fig:surgerydescriptionGOF(1;m)}. 
As the exterior of the $(2, 4)$-torus link is a Seifert fibered space over the annulus with one exceptional fiber of multiplicity $2$, and the regular fibers intersect the $0$- and $(-m)$-slopes 2 and $m+2$ times, respectively, the result of Dehn filling is the Seifert fibered space over the sphere with three exceptional fibers,
having Seifert invariants $\left(-1; (2, 1), (2, 1), (m+2, 1)\right)$. (See also Fig. \ref{fig:prism} the \nameref{section:Appendix}.) In particular then (see \cite{BalChiMacNiOchVafPMRP}), 
\begin{center}
$M(1;m) = \begin{cases}
L(4,-1)&(m=-1)\\
L(2,1)\, \sharp\, L(2,1)&(m=-2)\\
L(4,1)&(m=-3)\\
\mbox{a prism manifold}&(\mbox{otherwise}).
\end{cases}$
\end{center}

In the case of $GOF(-1;m)$ in $M(-1;m)$, which has a monodromy $h=(D_\bd)^{-1/2} \circ D_\alpha^m$, we have a surgery description by taking the mirror image of the case of $GOF(1;-m)$ in $M(1;-m)$. 
The manifold $M(-1;m)$ is homeomorphic 
to the Seifert fibered space having Seifert invariants $(1; (2, -1), (2, -1), (m-2, 1))$.
\end{proof}

\vspace{5mm}
\begin{figure}[h!]
\begin{center}
\includegraphics[scale=1]{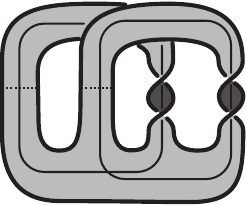}

\begin{picture}(400,0)(0,0)
\put(120,70){$\alpha$}
\put(157,70){$\beta$}
\put(215,20){$c_\alpha$}
\put(155,23){$c_\beta$}
\end{picture}
\caption{The trefoil in $S^3$ has monodromy $D_\alpha \circ D_\beta$.}
\label{fig:TrefoilFiber}
\end{center}
\end{figure}

\begin{figure}[h!]
\begin{center}
\includegraphics[scale=.6]{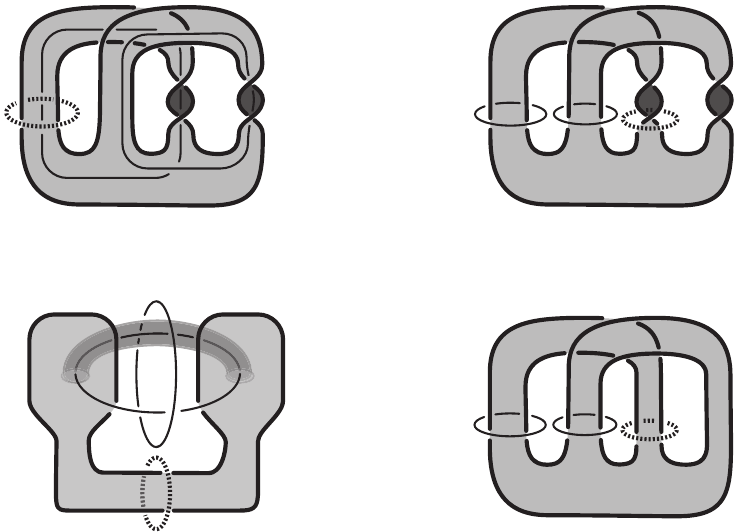}

\begin{picture}(400,0)(0,0)
\put(70,135){\scriptsize $L$}
\put(140,170){\scriptsize $c_1(\frac{k}{k-1})$}
\put(70,105){\scriptsize $c_2(\frac{\ell}{\ell-1})$}
\put(242,125){\scriptsize $\frac{k}{k-1}$}
\put(222,125){\scriptsize $\frac{\ell}{\ell-1}$}
\put(243,35){\scriptsize $-k$}
\put(221,35){\scriptsize $-\ell$}
\put(95,50){\scriptsize $-k$}
\put(120,85){\scriptsize $-\ell$}
\put(180,90){\scriptsize $(-1)$-twist along $c_1$ and $c_2$  $\downarrow $}
\put(185,135){$\sim$}
\put(185,45){$\sim$}
\end{picture}
\caption{$GOF(0;k,\ell)$ with monodromy $D_\beta^\ell\circ D_\alpha^k $ and an $\alpha$-loop in $M(0;k,\ell)$.}
\label{fig:surgerydescriptionGOF(0;k,l)}
\end{center}
\end{figure}

\begin{figure}[h!]
\begin{center}
\includegraphics[scale=.6]{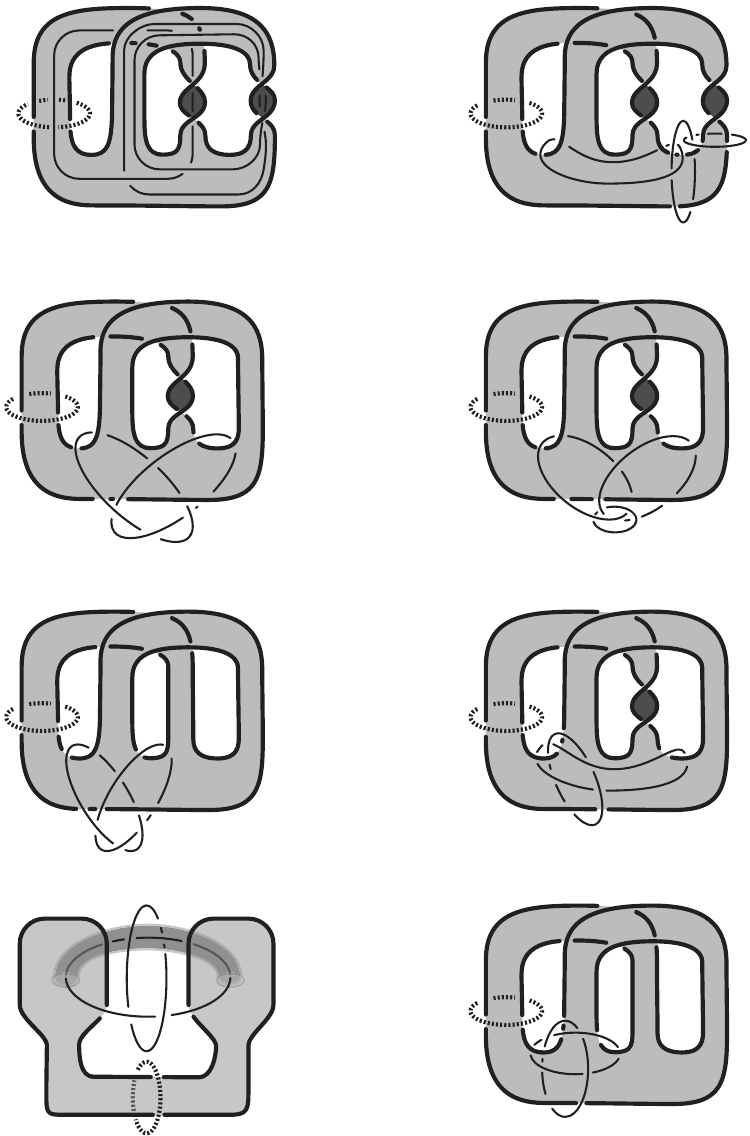}

\begin{picture}(400,0)(0,0)
\put(75,305){\scriptsize $L$}
\put(160,310){\scriptsize $c_1(2)$}
\put(97,285){\scriptsize $c_2(2)$}
\put(150,280){\scriptsize $c_3(\frac{m+3}{m+2})$}
\put(295,300){\scriptsize $2$}
\put(240,285){\scriptsize $2$}
\put(270,270){\scriptsize $\frac{m+3}{m+2}$}
\put(255,180){\scriptsize $1$}
\put(230,203){\scriptsize $1$}
\put(268,190){\scriptsize $-(m+3)$}
\put(97,203){\scriptsize $0$}
\put(75,190){\scriptsize $-(m+4)$}
\put(93,113){\scriptsize $0$}
\put(70,100){\scriptsize $-(m+4)$}
\put(248,95){\scriptsize $0$}
\put(265,110){\scriptsize $-m$}
\put(241,13){\scriptsize $0$}
\put(255,30){\scriptsize $-m$}
\put(120,83){\scriptsize $0$}
\put(140,50){\scriptsize $-m$}
\put(185,300){$\sim$}
\put(250,265){ $\downarrow$}
\put(185,265){\scriptsize $(-1)$-twist along $c_3$}
\put(175,220){ $\longleftarrow$}
\put(155,210){\scriptsize $(-1)$-twist along $c_2$}
\put(120,170){$\wr$}
\put(175,135){ $\longrightarrow$}
\put(165,125){\scriptsize $1$-twist along $c_1$}
\put(255,90){$\wr$}
\put(185,50){$\sim$}
\end{picture}
\caption{$GOF(1;m)$ with monodromy $D_\bd^{1/2}\circ D_\alpha^m$ and an $\alpha$-loop in $M(1;m)$.}
\label{fig:surgerydescriptionGOF(1;m)}
\end{center}
\end{figure}


Lemma \ref{lem:AllCrossingChangesAlongArcs}, then, provides an immediate corollary to Theorem \ref{thm:Manifolds}.

\ManifoldsForGOFKnotsWithGeneralizedCrossingChanges*

Recall from Theorem \ref{thm:Manifolds} that $GOF(0;k,\ell)$ and $GOF(\pm1;m)$ refer to the GOF-knots with monodromy $D_\beta^\ell\circ D_\alpha^k$ and $(D_\bd)^{1/2} \circ D_\alpha^m$. 

Since a (generalized) crossing change taking one GOF-knot to another must be around a crossing circle bounding a disk that intersects the fiber in an arc, the crossing change is a Dehn surgery along the curve formed by the union of the arc and its image. To see this, start with the arc $\alpha$ and its image $h(\alpha)$, sitting in a fiber $F$. Push them off of $F$ slightly (in the direction of $F \times \set{1}$ in the bundle). Fixing the endpoints, pull $h(\alpha)$ back through the monodromy, and the resulting arc sits just above $\alpha$, but on the other side of $F$. The two sub-arcs together form the crossing circle. Then, \cite{NiDSKPM} describes the way that the monodromy must change when the crossing change is performed. Combining this with Lemma \ref{lemma:FactorizationUnclean}, we have Theorem \ref{thm:CrossingChanges} immediately.

\section{Classes of clean or once-unclean arcs}
\label{section:ClassesCleanOnceUncleanArcs}
Let $F$ be a compact, oriented, connected surface with boundary, and let $h$ be an orientation preserving self-homeomorphism of $F$ 
such that the restriction of $h$ to the boundary is the identity.
We say two arcs properly embedded in $F$ are \emph{$h$-conjugate} or \emph{monodromy conjugate} if there exists 
an orientation preserving self-homeomorphism $f$ of $F$ which sends one arc to the other and commutes with $h$ (i.e., $f\circ h=h\circ f$). 
We also say two arcs in $F$ are \emph{$h$-equivalent} or \emph{monodromy equivalent} if the image of one arc by some power $h^n$ of $h$ is isotopic to the other arc.
By definition, two arcs are $h$-conjugate if they are $h$-equivalent.
\begin{lem}\label{lem:ConjugateArcs}
Let $K$ be a fibered knot with fiber $F$ and monodromy $h$.
If two arcs $\alpha$ and $\beta$ in $F$ are $h$-conjugate, then the two links $K\cup L_{\alpha}$ and $K\cup L_{\beta}$ are equivalent.
\end{lem}
\begin{proof}
Suppose that the arcs $\alpha$ and $\beta$ in $F$ are $h$-conjugate.
By definition, there exists an orientation preserving self-homeomorphism $f$ of $F$ such that $f(\alpha)=\beta$ and $f\circ\ h=h\circ f$. 
Since $f=h\circ\ f\circ\ h^{-1}$, the map $f:F\times \{0\}\to F\times \{0\}$ can be naturally extended into an orientation preserving self-homeomorphism of the surface bundle $(F\times [0,1])/h$, that is the exterior of $K$. 
For each point $x\in\bd F$, the loop $(\{x\}\times [0,1])/h$ is a meridian of $K$ and is invariant under the extended map. Hence it can be further extended into an orientation preserving self-homeomorphism of  the ambient $3$-manifold. 
This map shows that $K\cup L_{\alpha}$ and $K\cup L_{\beta}$ are equivalent.
\end{proof}

One GOF-knot might be transformed into multiple different GOF-knots by crossing changes when there are multiple clean or once-unclean arcs in a fiber. This would correspond, for instance, to a monodromy having multiple factorizations into the forms of Lemmas \ref{lemma:FactorizationClean} and \ref{lemma:FactorizationUnclean} 
(see Remark \ref{rem:Redundancy}). In this section, we will show that 
there are at most two $h$-conjugacy classes of arcs that are clean, or once-unclean, when there are two they can be realized disjointly in the fiber, and we will describe when this occurs.

This was proven for clean arcs in GOF-knots in \cite{CowLacUGOK}. We follow the arguments found in \cite{CowLacUGOK}, and modify them to the purpose of finding once-unclean arcs in addition to clean arcs.

Let $A$ be a directed arc in $\overline{\mathbb{H}^2}$, directed from endpoint $A_-$ to $A_+$, both in $\bd \mathbb{H}^2$. Following \cite{CowLacUGOK}, say that two distinct vertices of $\AC(F)$ are on the \emph{same side} of $A$ if their corresponding vertices in $\bd \mathbb{H}^2$ are not interleaved with the endpoints of $A$. If $x_1$ and $x_2$ are distinct points of $\bd \mathbb{H}^2$ on the same side of $A$, then $x_1 < x_2$ if $\{x_1, A_+\}$ and $\{x_2, A_-\}$ are interleaved. This defines a total order on points of one side of $A$.

Suppose $\alpha$ is an arc in $F$ that is clean with respect to the monodromy $h$. Let $\beta = h(\alpha)$. Either $\beta = \alpha$, in which case $h$ is a power of a Dehn twist (around $c_\alpha$, see Lemma \ref{lemma:FactorizationClean}), or there is an edge in the arc complex $\AC(F)$ between the vertices corresponding to $\alpha$ and $\beta$.

Now, consider the 1-complex $\alpha \cup \beta$, and let $N$ be a regular neighborhood of $\alpha \cup \beta$ in $F$. Also, consider the two resolutions of the interior intersection between $\alpha$ and $\beta$. One of the resulting 1-complexes will be connected, $R_1$, and the other will not, $R_2$. Let $N_i$ be a regular neighborhood in $F$ of the $R_i$ ($i = 1, 2$).
By Lemma \ref{lemma:FactorizationUnclean}, there are two cases. 

In Case (\ref{case:RVNC}), the frontier of $N_1$ consists of one closed loop (bounding a disk in $F$), and two arcs properly embedded in $F$. That these arcs will be essential and isotopic follows from the proof of Lemma \ref{lemma:FactorizationUnclean}; call either of them $\nu_1$. The frontier of $N_2$ consists of two closed loops (essential in $F$), and two arcs properly embedded in $F$. These arcs will also be essential and isotopic; call either of them $\nu_2$. The arcs $\nu_1$ and $\nu_2$ are disjoint from each other, and each can be made disjoint from both $\alpha$ and $\beta$. Then, in the arc complex, $\AC(F)$, there is a simplex $\Delta_\alpha$, whose vertices correspond to $\alpha, \nu_1$, and $\nu_2$, and there is a simplex $\Delta_\beta$, whose vertices correspond to $\beta, \nu_1$, and $\nu_2$. In particular, $\Delta_\alpha$ and $\Delta_\beta$ share an edge (the edge between the vertices corresponding to $\nu_1$ and $\nu_2$). In this case, we say that the vertices corresponding to $\alpha$ and $\beta$ are \emph{simplex-adjacent}, and call the edge between the vertices corresponding to $\nu_1$ and $\nu_2$ the \emph{common edge}. (Equivalently, we could say that there exist $2$-simplices associated with $\alpha$ and $\beta$ whose corresponding vertices in $\T$ are adjacent. In this case, the edge between these vertices in $\T$ corresponds to the common edge in $\AC(F)$.) 

In Case (\ref{case:RVPC}), two boundary curves of $N$ are embedded in the interior of $F$, and the remaining boundary curve intersects $\bd F$ in two arcs. In this case, the frontier of $N$ consists of two simple closed curves, and again two arcs. One of these arcs is inessential, as it is parallel into the boundary of $F$, and the other, call it $\nu$, is essential and disjoint from both $\alpha$ and $\beta$. Note in this case that, while $\alpha$ and $\beta$ intersect once, from the perspective of the arc complex $\alpha$ and $\beta$ represent vertices connected by an edge, because $\beta$ is properly isotopic to an arc disjoint from $\alpha$, so the vertices representing $\alpha$, $\beta$, and $\nu$ form a simplex in $\AC(F)$.

\begin{thm}
\label{thm:UsuallyAtMost2Classes}
If $h$ has at least two distinct $h$-equivalence classes of arcs which are either once-unclean or clean, then one of the following hold: 
\begin{enumerate}
\item \label{case:Identity} $h=Id$ (resp., $h=(D_{\bd})^{\pm1/2}$). 
In this case, all arcs are clean (fixed) (resp., once-unclean) and $h$-conjugate.
\item \label{case:PeriodicOrder2} 

$h^{\pm1}=D_\delta\circ D_\kappa^2=D_\kappa^2\circ D_\mu$, where $\kappa,\mu$ are disjoint arcs and $\delta=D_\kappa^2(\mu)$. In this case, $\kappa$ and $\mu$ are representatives of all $h$-equivalence classes of such arcs. 
\item \label{case:SquareDehnTwist} 
$h^{\pm1}=D_\kappa^2\circ D_\mu^0=D_\kappa^2\circ D_{\mu'}^0$, where $\kappa,\mu, \mu'$ are mutually disjoint arcs and $\mu'=D_\kappa(\mu)$. 
In this case, $\kappa,\mu$ and $\mu'$ are  
representatives of all $h$-equivalence classes of such arcs, and $\mu,\mu'$ are $h$-conjugate.  
\item \label{case:ReducibleBoundaryTwistSingleDehnTwist} $h^{\pm1}=D_\delta\circ D_\kappa^4=(D_\bd)^{1/2}\circ D_\mu^{-1}$, where $\kappa, \mu$ are disjoint arcs and $\delta=D^2_\kappa(\mu)$. In this case, $\kappa$ and $\mu$ are 
representatives of all $h$-equivalence classes of such arcs.
\item \label{case:ReducibleBoundaryTwistSquareDehnTwist} 
$h^{\pm1}=(D_\bd)^{1/2}\circ D_{\mu}^{-2}=D_\delta^2\circ D^2_{\mu'}=D_{\mu'}^2\circ D_{\mu''}^2$, where $\mu,\mu',\mu''$ are mutually disjoint arcs and $\mu''=D_\mu(\mu'),\delta=D_\mu^{-1}(\mu')$. In this case, $\mu,\mu'$ and $\mu''$ are representatives of all $h$-equivalence classes of such arcs, and $\mu',\mu''$ are $h$-conjugate. 
\item \label{case:PseduoAnosovTwoClean} 
$h^{\pm1}=D_\delta^{-1}\circ D_\kappa=D_\kappa\circ D_{\kappa'}^{-1}$, where $\kappa,\kappa'$ are disjoint arcs and $\delta=D_\kappa(\kappa')$. In this case, $\kappa$ and $\kappa'$ are representatives of all $h$-equivalence classes of such arcs. 
\item \label{case:PseudoAnosovOneCleanOneOnceUnclean} $h^{\pm1}=D_\delta^{-2}\circ D_\mu=D_\mu\circ D_\kappa^{-2}$, where $\kappa,\mu$ are disjoint arcs and $\delta=D_\mu(\kappa)$. In this case, $\kappa$ and $\mu$ are representatives of all $h$-equivalence classes of such arcs. 
\item \label{case:PseudoAnosovTwoOnceUnclean} $h^{\pm1}=D_\delta^{-2}\circ D_\mu^2=D_\mu^2\circ D_{\mu'}^{-2}$, where $\mu,\mu'$ are disjoint arcs and $\delta=D_\mu^2(\mu')$. In this case, $\mu$ and $\mu'$ are representatives of all $h$-equivalence classes of such arcs. 
\end{enumerate}
In particular, $h$ has at most two $h$-conjugacy classes of arcs which are either once-unclean or clean.
\end{thm}
\begin{rem}
In Theorem \ref{thm:UsuallyAtMost2Classes}, $\kappa$, $\kappa'$ (resp., $\mu,\mu',\mu''$) are clean arcs (resp., once-unclean arcs). 
The multiple factorizations of $h$ in (\ref{case:PeriodicOrder2})-(\ref{case:PseudoAnosovTwoOnceUnclean}) follow from the well-known formula $D_{f(\alpha)}=f\circ D_\alpha\circ f^{-1}$ for any automorphism $f$ and any arc $\alpha$. 
In case (\ref{case:SquareDehnTwist}), $f=D_\kappa$ commutes with $h^{\pm1}=D_\kappa^2$ and $f(\mu)=\mu'$, thus $\mu$ and $\mu'$ are $h$-conjugate.
Similarly, in case (\ref{case:ReducibleBoundaryTwistSquareDehnTwist}), $f=D_\mu$ commutes with $h^{\pm1}=(D_\bd)^{1/2}\circ D_\mu^{-2}$ and $f(\mu')=\mu''$, thus $\mu'$ and $\mu''$ are $h$-conjugate.
\end{rem}

\begin{proof}
If $h = Id$ (resp., $h=(D_{\bd})^{\pm1/2}$), it is clear that all arcs are clean (resp., once-unclean),
and that all arcs are $h$-conjugate, since any orientation preserving homeomorphism commutes with $h$.
 Otherwise, we consider separately the three homeomorphism types of the monodromy $h$. 

\noindent \emph{Case 1}. $h$ is periodic.

In this case, the induced automorphism of the tree $\T$ fixes a point in $\T$, which must either be a vertex, or a midpoint of an edge. 

If the fixed point is a vertex of $\T$, then this vertex corresponds to a $2$-simplex $\Delta$ in $\AC(F)$, and this corresponds to an ideal triangle in $\mathbb{H}^2$. So $\widetilde{h}$ induces a rotation of $\mathbb{H}^2$ around the center of this ideal triangle. Label the ideal vertices of this triangle $t_0, t_1,$ and $t_2$, so that $\widetilde{h}$ sends $t_i$ to $t_{i+1}$ (mod 3). Then each $t_i$ is a clean arc (in the same equivalence class, as observed in \cite{CowLacUGOK}). Consider, then, a vertex $v$ other than the vertices of the ideal triangle. Without loss of generality, we have $v < t_1 < \widetilde{h}(v)$ with respect to the edge from  $t_0$ to $t_2$. Then, $v$ and $\widetilde{h}(v)$ are on opposite sides of both the edge between $t_0$ and $t_1$, and the edge between $t_1$ and $t_2$, so they cannot be joined by an edge and cannot be simplex-adjacent. The first conclusion guarantees that $v$ is not clean, and the two conclusions together guarantee that $v$ is not once-unclean. Thus, there is one equivalence class of clean arcs, and no once-unclean arcs.

If, on the other hand, the fixed point is a midpoint of an edge of $\T$, say $e'$, then the automorphism fixes $e'$ set-wise, and acts as rotation around the midpoint. Let $e$ be the edge in $\AC(F)$ corresponding to $e'$. Let the endpoints of $e$ be $t_1$ and $t_2$. Then $e$ is an edge between two $2$-simplices $\Delta_1$ and $\Delta_2$ in $\AC(F)$, each with a third vertex, $v_1$ and $v_2$, respectively. Since the endpoints of $e'$ are interchanged by the automorphism, and $e$ is fixed, $v_1$ and $v_2$ must be interchanged. Note that since $\widetilde{h}$ is elliptic, $t_1$ and $t_2$ cannot be fixed, so they must be interchanged as well. Thus, $t_1$ and $t_2$ correspond to a single monodromy equivalence class of clean arcs, and $v_1$ and $v_2$ are simplex-adjacent with common edge $e$, so they correspond to a monodromy equivalence class of once-unclean arcs. For any vertex $v$ off of $e$, $v$ and $\widetilde{h}(v)$ are on opposite sides of $e$, so they cannot be joined by an edge, and the only way that $v$ and $\widetilde{h}(v)$ could be simplex-adjacent would be if $e$ were the common edge between them. So, there is exactly one $h$-equivalence class of clean arcs, as observed in \cite{CowLacUGOK}, and there is exactly one $h$-equivalence class of once-unclean arcs, and these two classes can be realized disjointly since, say, $t_1$ and $v_1$ are joined by an edge. 
The periodic monodromies from Table \ref{table:ClassificationMonodromies} with order 2 are $D_\delta \circ D_\kappa^2$ and $D_\delta^2 \circ D_\mu$, which are conjugate, and each admits one clean arc and one once-unclean arc. In fact, putting $\delta=D_\kappa^2(\mu)$ for disjoint arcs $\kappa,\mu$, we have 
$$D_\delta\circ D_\kappa^2=(D_\kappa^2\circ D_\mu\circ D_\kappa^{-2})\circ D_\kappa^2=D_\kappa^2\circ D_\mu.$$
Hence we have conclusion (\ref{case:PeriodicOrder2}).

\noindent \emph{Case 2}. $h$ is reducible. 

Then $h$ leaves an essential arc in $F$ fixed, up to isotopy, and $\widetilde{h}$ is parabolic. Hence, as an element of $SL(2, \mathbb{Z})$, $\widetilde{h}$ is conjugate to 
\[ \left( \begin{array}{cc} \pm 1 & n \\ 0 & \pm 1 \end{array} \right),\]
with $n \in \mathbb{Z} \rmv \set{0}$. For any homeomorphism $f$ of $F$, \[| int(\alpha) \cap int(h(\alpha)) | = | int(f(\alpha)) \cap int(f(h(\alpha))) |,\] so $\alpha$ is clean or once-unclean with respect to $h$ if and only if $f(\alpha)$ is clean or once-unclean, respectively, with respect to $f\circ h \circ f^{-1}$. Thus, conjugation will not affect the properties in the theorem, so we may assume that $\widetilde{h}$ is given by the matrix above.

Thus $h = (D_\alpha)^n$ or $(D_\bd)^{\pm 1/2} \circ (D_\alpha)^n$ for some $n \in \mathbb{Z} - \set{0}$, up to free isotopy, where $\alpha$ is the arc represented by $1/0$, fixed by $D_\alpha$. 

Let us first consider $h = (D_\alpha)^n$. An arc represented by $p/q$ will be sent to the arc represented by $(p \pm nq)/q$. The arc represented by $1/0$ is, in fact, fixed by $h$, regardless of the value of $n$. The only way there can exist a second clean arc is if $n = \pm 1$, in which case, all arcs represented by an integer are clean and $h$-equivalent, and there are no once-unlcean arcs. Otherwise, an arc is once-unclean only if $| pq - q(p \pm nq)| = |nq^2| = 2$, so $q = \pm 1$, $n = \pm 2$, and there are two $\widetilde{h}$-equivalence classes of vertices corresponding to once-unclean arcs, represented disjointly by the arcs corresponding to $1/1$ and $2/1$. Thus, $D_\kappa^2 = D_\kappa^2 \circ D_\mu^0 = D_\kappa^2 \circ D_{\mu'}^0$ for the two disjoint once-unclean arcs $\mu$ and $\mu'$. The homeomorphism $f=D_{\kappa}$ sends $\mu$ to $\mu'$ (or vice versa), and commutes with the monodromy $h$,
so the $h$-conjugacy class of once-unclean arcs is unique. This is conclusion (\ref{case:SquareDehnTwist}).

On the other hand, to understand $h = (D_\bd)^{1/2} \circ D_\alpha^n$, let us consider the effect of composing $(D_\bd)^{1/2}$ with $D_\alpha^n$.  
An arc that is fixed by $D_\alpha^n$ will be once-unclean with respect to $h$. 
A clean left-veering arc (hence $n<0$) will become a clean right-veering arc with respect to $h$, 
but a clean right-veering arc (hence $n>0$) will have two intersections with its image when composed with $(D_\bd)^{1/2}$ as the boundary twist reinforces the veer. 
A once-unclean left-veering arc (hence $n<0$) will become a once-unclean right-veering arc with respect to $h$, 
but a once-unclean right-veering arc (hence $n>0$) will have three intersections with its image when composed with $(D_\bd)^{1/2}$ as the boundary twist reinforces the veer even more.
Arcs that are neither fixed, clean, nor once-unclean with respect to $(D_\alpha)^n$ will be neither fixed, clean, nor once-unclean with respect to $h$. Hence, the number of classes of arcs that are either fixed, clean, or once-unclean with respect to $(D_\bd)^{1/2} \circ D_\alpha^n$ is less than or equal to the number of such classes for $(D_\alpha)^n$. The analysis of how clean and once-unclean arcs with respect to $(D_\alpha)^n$ extends to $h$ is summarized in Table \ref{table:ComposingBoundaryTwist}.

\begin{table}[h!]
  \begin{center}
    \begin{tabular}{c|c} 
      Monodromy & $h$-equivalence classes of clean or once-unclean arcs \\
      \hline
      $(D_\bd)^{1/2} \circ D_\alpha$ & 1 once-unclean arc\\
      $(D_\bd)^{1/2} \circ D_\alpha^{-1}$ & 1 clean arc, 1 once-unclean arc\\
      $(D_\bd)^{1/2} \circ D_\alpha^2$ & 1 once-unclean arc\\
      $(D_\bd)^{1/2} \circ D_\alpha^{-2}$ & 3 once-unclean arcs\\
      $(D_\bd)^{1/2} \circ D_\alpha^n, |n| > 2$ & 1 once-unclean arc\\
    \end{tabular}
        \caption{A summary of the clean and once-unclean $h$-equivalence classes with respect to $(D_\bd)^{1/2} \circ D_\alpha^n$.}
    \label{table:ComposingBoundaryTwist}
  \end{center}
\end{table}

Consider the two cases of Table \ref{table:ComposingBoundaryTwist} with more than one class of arcs that are clean or once-unclean.

For the monodromy $(D_\bd)^{1/2} \circ D_{\alpha}^{-1}$, putting $\delta=D_\kappa^2(\mu)$ for disjoint arcs $\kappa,\mu$, we have
\begin{align*}
D_\delta\circ D_\kappa^4 = &\ (D_\kappa^2\circ D_\mu\circ D_\kappa^{-2})\circ D_\kappa^4\\
 = &\ D_\kappa \circ (D_\kappa \circ D_\mu \circ D_\kappa) \circ D_\kappa \circ D_\mu \circ D_\mu^{-1} \\
 = &\  D_\kappa \circ (D_\mu \circ D_\kappa \circ D_\mu) \circ D_\kappa \circ D_\mu \circ D_\mu^{-1} \\
 = &\ (D_\bd)^{1/2}\circ D_\mu^{-1},
\end{align*}
where the third step follows from the braid relation, since the curves $c_\kappa$ and $c_\mu$ intersect once.

Then we have conclusion (\ref{case:ReducibleBoundaryTwistSingleDehnTwist}).

The only remaining reducible monodromy, then, admitting multiple $h$-equivalence classes of clean or once-unclean arcs, $(D_\bd)^{1/2} \circ D_{\alpha}^{-2}$, with three once-unclean arcs (one of which is $\alpha$), must be able to be re-factorized as a product of two squares of Dehn twists, as in Table \ref{table:ClassificationMonodromies}. In fact, one can verify that $(D_\bd)^{1/2}\circ D_{\mu}^{-2}=D_{\mu'}^2\circ D_{\mu''}^2$, where $\mu,\mu',\mu''$ are mutually disjoint arcs, each once-unclean with respect to the monodromy, and $\mu'' = D_{\mu}(\mu')$. Further, the homeomorphism $f=D_\mu$ sends $\mu'$ to $\mu''$, and commutes with the monodromy, so there are only two $h$-conjugacy classes of once-unclean arcs.
Thus we have conclusion (\ref{case:ReducibleBoundaryTwistSquareDehnTwist}).

\noindent \emph{Case 3}. $h$ is pseudo-Anosov.

In this case, the induced action on $\mathbb{H}^2$ is loxidromic, having two fixed points on $\bd \mathbb{H}^2$. The fixed points cannot be vertices of $\AC(F)$, because $h$ is not reducible. A loxidromic mapping class has a set-wise fixed axis, which we will call $A$, whose endpoints $A_-$ and $A_+$ on $\bd \mathbb{H}^2$ are the fixed points of $\widetilde{h}$, where $A_-$ is a repelling point, and $A_+$ is an attracting point.

Then $A$ is an edge directed from $A_-$ to $A_+$, which defines a total order on $\bd \mathbb{H}^2$ on either side of $A$, as defined above, which is preserved by $\widetilde{h}$.

Again from \cite{CowLacUGOK}, a vertex $v \in \AC(F)$ is said to be \emph{visible from $A$} if $v$ is adjacent in $\AC(F)$ to a vertex on the opposite side of $A$.

Coward and Lackenby showed that if a vertex and its image are joined by an edge, then both vertices must be visible from $A$, and that there are at most two equivalence classes of such vertices, one for each side of $A$. 
We will show also that if $v$ and $\widetilde{h}(v)$ are simplex-adjacent, then $v$ is visible from the axis $A$. Suppose that $v$ is not visible from $A$. Then, every simplex of which $v$ is a vertex has all vertices on the same side of $A$. Further, since $v$ is not visible from $A$, there exists such a simplex with vertices $v, v_-$, and $v_+$, so that $v_- < v < v_+$. Now, $v_+ < \widetilde{h}(v_+)$, since $\widetilde{h}$ moves points along the circle $\bd \mathbb{H}^2$ away from $A_-$ and towards $A_+$. Since $\{v_-, v_+\}$ and $\{\widetilde{h}(v_-), \widetilde{h}(v_+)\}$ both form the endpoints of edges, they cannot be interleaved. In this case, in order for $v$ and $\widetilde{h}(v)$ to be simplex-adjacent, it would need to be true that the edge between $v_-$ and $v_+$ is the common edge of $v$ and $\widetilde{h}(v)$. But then we would have $\widetilde{h}(v_-) < \widetilde{h}(v) < \widetilde{h}(v_+)$, since $\widetilde{h}$ preserves the order, so the edge between $v_-$ and $\widetilde{h}(v)$ precludes the existence of an edge between $\widetilde{h}(v_-)$ and $\widetilde{h}(v_+)$. Thus, it is impossible for $v$ and $\widetilde{h}(v)$ to be simplex-adjacent.

Next, we will show that there are at most two $h$-equivalence classes of arc that are either clean or once-unclean in $F$. To prove this, we will construct a fundamental domain for the action of $\widetilde{h}$ on $A$, and examine the points visible from such a domain. 
Suppose $v$ and $\widetilde{h}(v)$ are simplex-adjacent, and are therefore both visible from $A$. Call the common edge $e$, and call its endpoints $x$ and $y$. Then $e$ separates $v$ and $\widetilde{h}(v)$, but $x$ and $y$ cannot be on the same side of $A$ since both $v$ and $\widetilde{h}(v)$ are visible from $A$. Say $x$ is on the same side of $A$ as $v$, and $y$ is on the opposite side. Because there is an edge between $v$ and $y$, there is also an edge between $\widetilde{h}(v)$ and $\widetilde{h}(y)$. Then a fundamental domain for the action of $\widetilde{h}$ on $A$ is the interval between the edge from $v$ to $y$ and the edge from $\widetilde{h}(v)$ to $\widetilde{h}(y)$. This interval is divided into three sub-intervals by the edge $e$ (from $x$ to $y$) and the edge from $\widetilde{h}(v)$ to $y$. For each point in the first sub-interval, the only vertices visible on the same side as $v$ are $v$ and $x$; for each point in the second sub-interval, the only vertices visible on the same side as $v$ are $x$ and $\widetilde{h}(x)$; for each point the third sub-interval, the only vertex visible on the same side as $v$ is $\widetilde{h}(v)$. Thus, $v$ and $x$ are the only vertices in distinct $\widetilde{h}$-equivalence classes that are visible from $A$ on this side of $A$. Now, if $x$ and $\widetilde{h}(x)$ were simplex-adjacent, the common edge between them would have to be the edge between $\widetilde{h}(v)$ and $y$. This would imply that $y = \widetilde{h}(y)$, which is impossible as $h$ is pseudo-Anosov. Hence, $v$ determines a unique $\widetilde{h}$-equivalence class of vertices on one side of $A$ corresponding to once-unclean arcs. 

Further, the edge from $\widetilde{h}(v)$ to $y$ prevents $x$ from being joined to $\widetilde{h}(x)$ by an edge, so that $x$ does not represent a clean arc. So, if there is a vertex on one side of $A$ that is simplex-adjacent to its image under $\widetilde{h}$, then there is not a vertex on the same side of $A$ representing a clean arc.

Thus, either $F$ has at most two $h$-equivalence classes of clean arcs, and no once-unclean arcs; at most two $h$-equivalence classes of once-unclean arcs, and no clean arcs; or at most one $h$-equivalence class of clean arcs and at most one $h$-equivalence class of once-unclean arcs, and in all cases the arcs can be realized disjointly.

Finally, if we can identify an arc that is clean, $\kappa$, its image $h(\kappa)$, and two arcs $\delta$ and $\delta'$ that are each disjoint from $\kappa$ and $h(\kappa)$ that intersect each other once, then the monodromy has a second arc that is either clean or once-unclean exactly when one of $\delta$ or $\delta'$ is clean or once-unclean. 

Similarly, if we can identify an arc that is once-unclean $\mu$, its image $h(\mu)$, and two arcs $\delta$ and $\delta'$ that are each disjoint from $\mu$ and $h(\mu)$ and from each other, then the monodromy has a second arc that is either clean or once-unclean exactly when one of $\delta$ or $\delta'$ is clean or once-unclean. 

Thus, for each pseudo-Anosov monodromy in Table \ref{table:ClassificationMonodromies}, we can identify the required arcs $\delta$ and $\delta'$. There are two $h$-equivalence classes of arcs that are both clean for $D_\delta^{-1} \circ D_\kappa$, where $\delta$ can be seen to be clean, giving rise to conclusion (\ref{case:PseduoAnosovTwoClean}). There are $h$-equivalence classes of arcs that are clean and once-unclean for $D_{\delta}^{-2} \circ D_\mu$, where $\delta$ can be seen to be clean, giving rise to conclusion (\ref{case:PseudoAnosovOneCleanOneOnceUnclean}). And there are two $h$-equivalence classes of arcs that are once-unclean for $D_\delta^{-2} \circ D_\mu^2$, where $\delta$ can be seen to be once-unclean, giving rise to conclusion (\ref{case:PseudoAnosovTwoOnceUnclean}).
\end{proof}

Corollary \ref{cor:GOFequivalence} now follows immediately from Theorem \ref{thm:UsuallyAtMost2Classes}, Lemma \ref{lem:ConjugateArcs}, and Theorem \ref{thm:Manifolds}.


\section*{Appendix}
\label{section:Appendix}

Fig. \ref{fig:euivL(4,1)} shows us the two knots $GOF(-1;1)$ and $GOF(0;-4,-1)$ are equivalent in $L(4,1)$.
\begin{figure}[h!]
\begin{center}
\includegraphics[scale=.8]{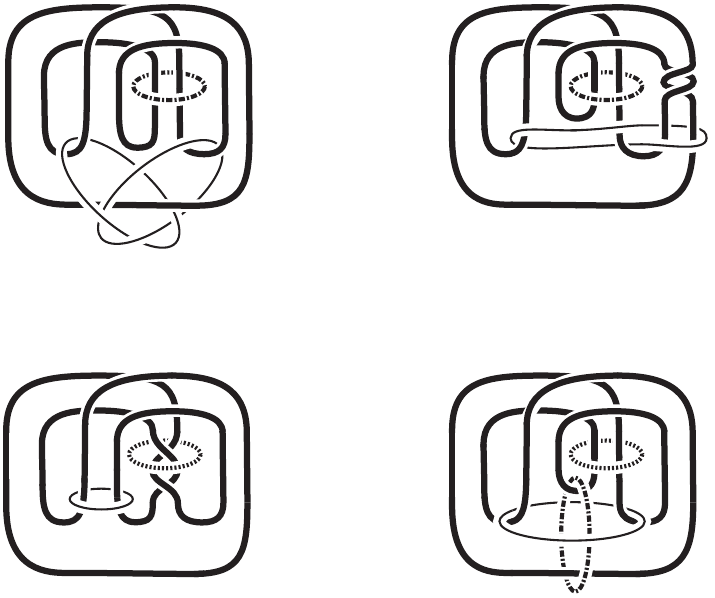}

\begin{picture}(400,0)(0,0)
\put(70,135){$GOF(-1;1)$}
\put(60,10){$GOF(0;-4,-1)$}
\put(125,150){\scriptsize $0$}
\put(70,150){\scriptsize $-1$}
\put(260,180){\scriptsize $4$}
\put(250,30){\scriptsize $4$}
\put(85,40){\scriptsize $4$}
\put(175,180){$\longrightarrow$}
\put(265,130){$\wr$}
\put(175,60){$\sim$}
\end{picture}
\caption{The GOF-knots $GOF(-1;1)$ and $GOF(0;-4,-1)$ in $L(4,1)$ are equivalent.
Compare with Fig. \ref{fig:crossingchangeL(4,1)}.}
\label{fig:euivL(4,1)}
\end{center}
\end{figure}
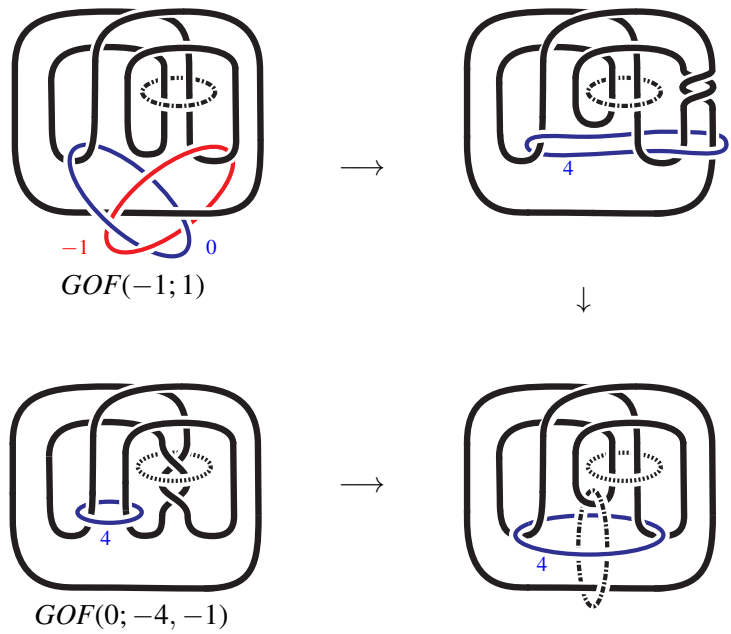

Fig. \ref{fig:equivL(2,1)L(2,1)} shows us the two knots $GOF(1;-2)$ and $GOF(0;-2-2)$ are equivalent in $L(2,1)\ \#\ L(2,1)$.

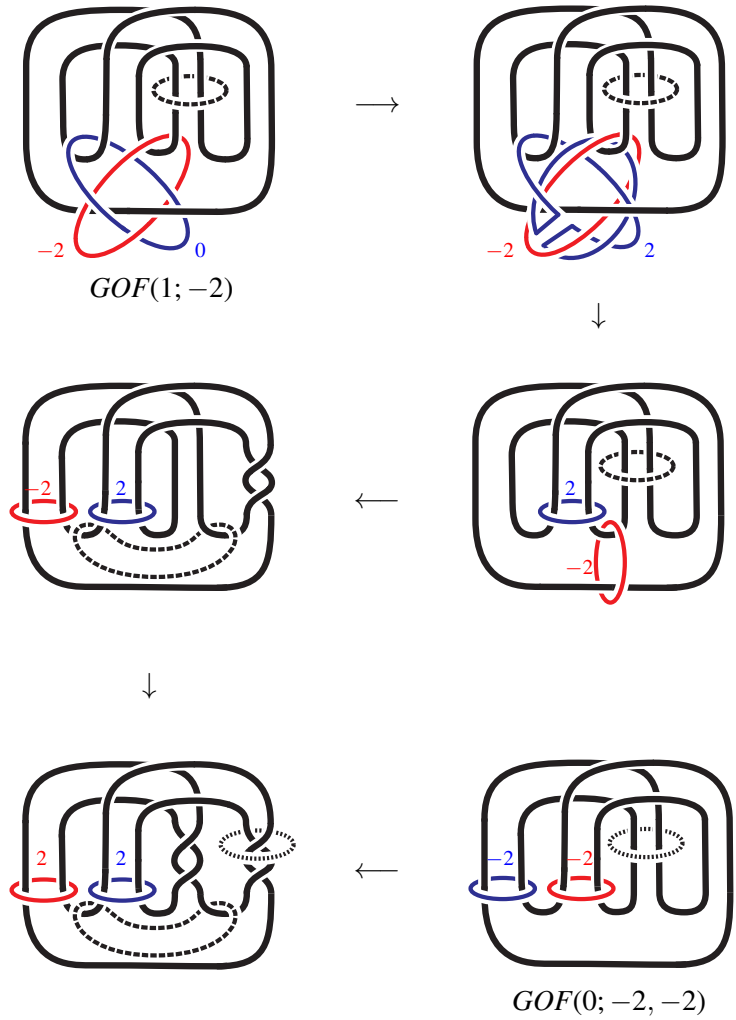
\begin{figure}[h!]
\begin{center}
\includegraphics[scale=.8]{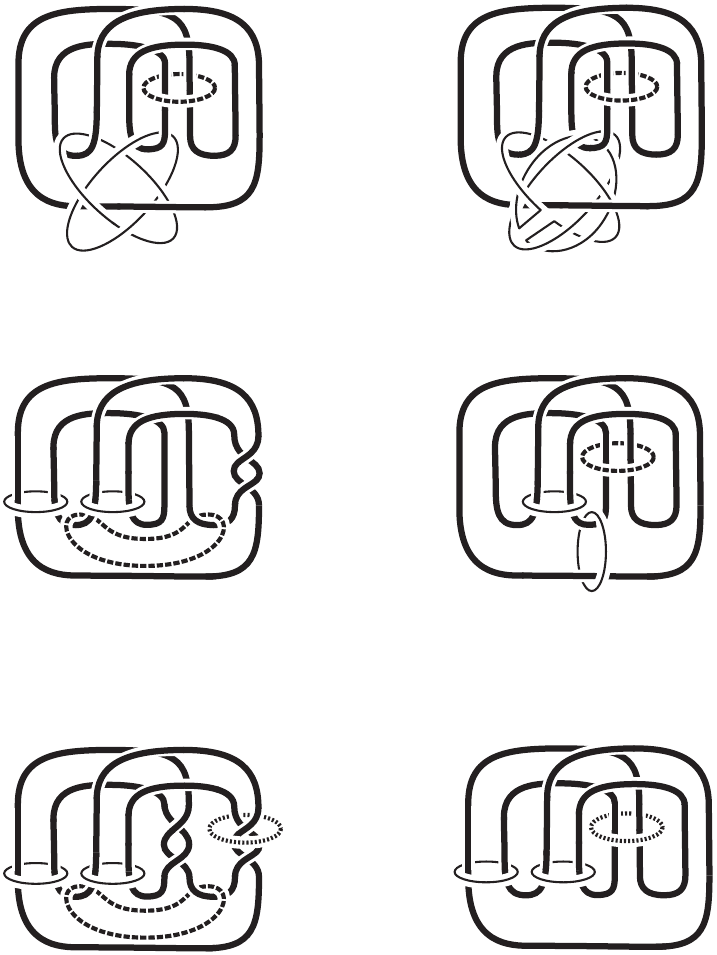}

\begin{picture}(400,0)(0,0)
\put(80,270){$GOF(1;-2)$}
\put(240,0){$GOF(0;-2,-2)$}
\put(120,285){\scriptsize $0$}
\put(60,285){\scriptsize $-2$}
\put(290,285){\scriptsize $2$}
\put(230,285){\scriptsize $-2$}
\put(260,198){\scriptsize $2$}
\put(260,170){\scriptsize $-2$}
\put(90,198){\scriptsize $2$}
\put(57,198){\scriptsize $-2$}
\put(90,55){\scriptsize $2$}
\put(60,55){\scriptsize $2$}
\put(230,55){\scriptsize $-2$}
\put(260,55){\scriptsize $-2$}
\put(180,340){$\longrightarrow$}
\put(270,260){$\wr$}
\put(180,190){$\sim$}
\put(100,120){$\downarrow$}
\put(180,50){$\longleftarrow$}
\end{picture}
\caption{The GOF-knots $GOF(1;-2)$ and $GOF(0;-2-2)$ in $L(2,1)\ \#\ L(2,1)$ are equivalent: 
The fifth diagram in Fig. \ref{fig:surgerydescriptionGOF(1;m)} is used for representing $GOF(1;-2)$.
Compare with Fig. \ref{fig:crossingchangeL(2,1)L(2,1)}.}
\label{fig:equivL(2,1)L(2,1)}
\end{center}
\end{figure}

Fig. \ref{fig:prism} discribes the $3$-manifold $M(1;m)$ by the Kirby calculus.
\begin{figure}[h!]
\begin{center}
\includegraphics[scale=.8]{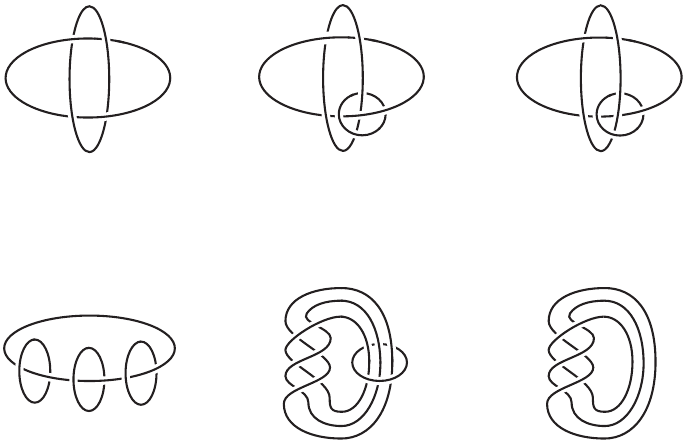}

\begin{picture}(400,0)(0,0)
\put(87,185){\scriptsize $0$}
\put(100,170){\scriptsize $-m$}
\put(185,185){\scriptsize $0$}
\put(200,170){\scriptsize $-m$}
\put(202,128){\scriptsize $\frac{1}{0}$}
\put(280,185){\scriptsize $-1$}
\put(298,170){\scriptsize $-m-1$}
\put(297,128){\scriptsize $-1$}
\put(255,20){\scriptsize $-1$}
\put(238,36){\scriptsize $-m-1$}
\put(255,48){\scriptsize $-1$}
\put(153,20){\scriptsize $-1$}
\put(136,36){\scriptsize $-m-1$}
\put(153,48){\scriptsize $-1$}
\put(213,43){\scriptsize $\frac{1}{0}$}
\put(105,20){\scriptsize $-2$}
\put(70,20){\scriptsize $-m-2$}
\put(55,20){\scriptsize $-2$}
\put(80,68){\scriptsize $-1$}
\put(130,150){$\longrightarrow$}
\put(230,150){$\longrightarrow$}
\put(285,100){$\wr$}
\put(230,50){$\longleftarrow$}
\put(130,50){$\longleftarrow$}
\end{picture}
\caption{The $3$-manifold in which the GOF-knot $GOF(1;m)$ sits:
The result of a surgery along the $(2,4)$-torus link with surgery coefficient $0$ and $-m$ is homeomorphic to the Seifert fibered space with Seifert invariants $(-1; (2, 1), (2, 1), (m+2, 1))$.
Compare with Figure 2 in \cite{BalChiMacNiOchVafPMRP}.
}
\label{fig:prism}
\end{center}
\end{figure}

%
\clearpage
%
%
%

\end{document}